\documentclass[10pt,reqno]{amsart}
\usepackage[english]{babel}
\usepackage{graphicx}
\usepackage[numbers]{natbib}
\usepackage[colorlinks=true, allcolors=blue]{hyperref}
\usepackage{amsmath,amsthm,amssymb,amsfonts}
\usepackage{mathrsfs,xfrac,nicefrac,faktor,esint,dsfont}
\usepackage{enumitem}
\usepackage{tikz-cd}
\usepackage{accents}
\usepackage{url}
\usepackage{xparse}
\usepackage{booktabs} 
\usepackage{tabularx} 
\usepackage{xcolor}
\allowdisplaybreaks[4]
\newtheorem{theorem}{Theorem}[section]
\newtheorem{proposition}[theorem]{Proposition}
\newtheorem{lemma}[theorem]{Lemma}

\theoremstyle{definition}
\newtheorem{definition}[theorem]{Definition}
\newtheorem{remark}[theorem]{Remark}

\numberwithin{equation}{section}

\newcommand{\T}{\mathbb{T}}
\renewcommand{\d}{{\mathrm{d}}}
\newcommand{\R}{\mathbb{R}}

\newcommand{\Hs}{{H}^{-s}_x}
\DeclareMathOperator{\supp}{supp}

\title[NONUNIQUENESS of NLW]{Sharp Non-uniqueness of weak solutions to the semilinear wave equation} 
\author[P. Qu]{Peng Qu}
\address{School of Mathematical Sciences \& Shanghai Key Laboratory for Contemporary Applied Mathematics, Fudan University, China.}
\email{pqu@fudan.edu.cn}

\author[M.X. Zhang]{Mingxin Zhang$^*$}
\address{School of Mathematical Sciences, Fudan University, Shanghai 200433, China.}
\email{mxzhang24@m.fudan.edu.cn}
\thanks{*Corresponding author. E-mail address: mxzhang24@m.fudan.edu.cn.}


\begin{document}

\begin{abstract}
In this paper, we establish a sharp non-uniqueness result to the semilinear wave equation 
\[
    \partial_{tt}u-\partial_{xx} u= \Theta_1(\partial_t u)^2 + \Theta_2(\partial_x u)^2.
\]
For any $\Theta_1 \Theta_2 < 0$, $\beta<1$, we construct infinitely many weak solutions $u \in C_{t,x}^{\beta}$ with zero initial data. 
To our knowledge, this is the first non-uniqueness result for a semilinear wave equation. 
At this low level of regularity $C_{t,x}^{\beta}(\beta<1)$, the quadratic term is defined by Littlewood--Paley decompositions in $C_t\Hs$. 
Our construction is based on a frequency-localized convex integration scheme. 
Together with the classical local uniqueness theory in \(C_{t,x}^{1}\), 
this non-uniqueness result is sharp.
\end{abstract}

\maketitle

\section{Introduction}
We consider the quadratic semilinear wave equation with initial data  
\begin{equation}\label{eq:NLW}
	\left\{\aligned
	&  \partial_{tt}u-\partial_{xx} u=\Theta_1(\partial_t u)^2 + \Theta_2(\partial_x u)^2,\\
    &   u(0,\cdot)=u_0,\qquad\partial_tu(0,\cdot)=u_1.
	\endaligned
	\right.
\end{equation}
where the unknown scalar function $u:[0,T]\times\T\to\R$ denotes the wave amplitude, 
$\partial_t u(t,x)$ denotes the wave velocity, 
where $\mathbb{T}=\mathbb{R}/\mathbb{Z}$. 
This kind of semilinear wave equation plays an important role in mathematics and physics, 
and its classical solutions and low regularity weak solutions have been widely studied.  
$\Theta_1, \Theta_2 \in \mathbb{R}$ are two constant parameters, which are assumed to satisfy $\Theta_1 \Theta_2 < 0$ throughout this paper. 
There also exist many important works considering the equation in multi-dimensional cases as 
 \begin{equation}\label{eq:NLW2}
	\left\{\aligned
	&  \partial_{tt}u-\Delta u=\Theta_1(\partial_t u)^2 + \Theta_2|\nabla u|^2,\\
    &   u(0,\cdot)=u_0\in H^s,\qquad\partial_tu(0,\cdot)=u_1\in H^{s-1}.
	\endaligned
	\right.
\end{equation}
Roughly speaking, our non-uniqueness construction in this paper can be adapted to these multi-dimensional cases by using a simple extension.  
\subsection{Background on the quadratic semilinear wave equation.}
In this part, we review the results for semilinear wave equations
with nonlinear terms of quadratic derivatives. 
For quadratic derivatives, the scaling critical Sobolev
regularity is $H^{s_c}$ with 
$$
s_c=\frac{n}{2}.
$$
Another threshold arising from the light
cone is $H^{s_l}$ with 
$$
s_l=\frac{n+5}{4}.
$$
Lindblad \cite{Lindblad1993,Lindblad1996} and Fang and Wang \cite{FangWang2005} give a discussion
of these two regularity thresholds for the
well-posedness and ill-posedness.

For the case of three dimensions, Ponce and Sideris \cite{PonceSideris1993}
proved local well-posedness for nonlinear wave equations with quadratic nonlinearities
when the nonlinearity satisfies the null condition. 
While Zhou \cite{Zhou2003Wave} considered the equations \eqref{eq:NLW} with $\Theta_1=0$ 
who establishes local well-posedness in
the range
$$
s_l<s<\frac52.
$$
with $n=4$.
Tataru \cite{Tataru1999} considered the following equations 
\begin{equation}\label{eq:tataru-model}
\square u = G(u)|\nabla u|^2 \notag 
\end{equation}
and proved that for $n\geq5$, the Cauchy problem is locally
well-posed in
$
H^s(\mathbb R^n)\times H^{s-1}(\mathbb R^n)
$
for every
$
s>s_c.
$
On the uniqueness, Sideris \cite{Sideris1981Thesis} proved the uniqueness of the semilinear wave equations in $C_t H_x^{1+}$, 

Moreover, 
Hidano, Jiang, Lee and Wang \cite{HidanoJiangLeeWang2020} proved unconditional local well-posedness for the three-dimensional quadratic semilinear wave equation with radial initial data in $H^s\times H^{s-1}$ for every $s>s_c$. 
Liu and Wang \cite{LiuWang2021} who considered equations of the form
\begin{equation}\label{eq:LiuWang}
\square u
= 
C^{\alpha\beta}
\partial_\alpha u\,\partial_\beta u,\notag
\end{equation}
gave a comprehensive discussion on the well-posedness and ill-posedness. 

As shown above, the failure of local
well-posedness may arise from the failure of continuous dependence,
blow up and other phenomena, 
but the results about the non-uniqueness of weak solutions are few. 
To our knowledge, our paper is the first non-uniqueness result about the semilinear wave equation. 

\subsection{Overview of the convex integration scheme.}
Thanks to the seminal works of De Lellis and Sz\'{e}kelyhidi
\cite{dls09,DS13}, many results in 
non-uniqueness of weak solutions to fluid dynamics systems emerge during the past two
decades. Inspired by Nash's $C^1$ isometric embedding theorem
\cite{Nash1954}, they developed convex integration for the incompressible
Euler equations and constructed highly irregular weak solutions.
Subsequently, Isett \cite{I18} proved the flexibility part of Onsager's
conjecture for the three-dimensional Euler equations. 
More applications can see \cite{DeLellisKwon2022,GiriKwon2022,GKN24,GKN26}.

Another major breakthrough was obtained by Buckmaster and Vicol
\cite{bv19b}, who proved the non-uniqueness of weak solutions to the
three-dimensional Navier--Stokes equations in $C_tL_x^2$.
Their work introduced intermittent convex integration, which allows the
viscous term to be controlled through spatial intermittency. Cheskidov
and Luo \cite{cl20.2} further developed this method by introducing
temporal intermittency and obtained sharp non-uniqueness results for the
incompressible Navier--Stokes equations. These
ideas have since been applied to many related problems, including the
Navier--Stokes type equations \cite{CZZ26,lqzz24,lq20,lt20,MNY26}, transport
equations \cite{cl21,cl22,ms18}, stationary Navier--Stokes equations
\cite{luo19,Zhao26}, magnetohydrodynamic equations \cite{bbv20,lqzz22,MNY25},  
surface quasi-geostrophic equation \cite{BSV19,ChengKwonLi2021,DGR24,IL24,Zhao24}.

Most of the non-uniqueness results were obtained 
in super-critical space for equations of fluid dynamics.
Meanwhile, methods for critical and sub-critical spaces and for other kinds of PDEs 
are developed recently. 
An early example was given by Christ
\cite{CM05}, who constructed non-trivial generalized solutions of the
two-dimensional periodic Navier--Stokes equations in
$
C^0([0,1];H^s(\mathbb T^2)),s<0,
$
with zero initial data. 
Christ \cite{CS05} also established non-uniqueness of generalized
solutions to the one-dimensional periodic nonlinear Schr\"odinger
equation with cubic or quadratic nonlinearities in
$
C^0([0,1];H^s(\mathbb T)),s<0.
$
More recently, Coiculescu and Palasek \cite{MS26} proved non-uniqueness
for the 3-D Navier--Stokes equations in
the critical space $\mathrm{BMO}^{-1}$. 
Cheskidov, Dai and Palasek \cite{CDP25} give a more stronger result about this.
Ashkarian, Bhargava 
Gismondi and Novack \cite{ABGM25} constructed non-trivial stationary
solutions to the two-dimensional Navier--Stokes equations satisfying in
$
\bigcap_{\varepsilon\in(0,1)}
L^{2-\varepsilon}(\mathbb T^2)
\cap
\dot H^{-\varepsilon}(\mathbb T^2).
$
Fujii
\cite{MI26} obtained a sharp classification of uniqueness and
non-uniqueness of mild solutions in the critical Besov spaces
$
\dot B_{q,r}^{-1+\frac{d}{q}}(\mathbb R^d),
$
and constructed infinitely many global solutions from zero initial data. 
Cheskidov and Hou
\cite{CH26} proved the failure of unconditional uniqueness in every
Besov space
$
B_{q,r}^{-\theta}(\mathbb T^d),
\theta>0, d\geq2.
$
Qu and Zhang \cite{QZ26} proved the non-uniqueness of stationary weak solutions to the three-dimensional fractional MHD equations in critical or sub-critical Besov spaces. 
Gismondi and Radu \cite{NA25} constructed non-trivial
stationary solutions to the dissipative surface quasi-geostrophic
equation in 
$
\dot H^{-1-\varepsilon}(\mathbb T^2)
\text{ for every }\varepsilon>0,
$
Gismondi, Ma, Pathak and Radu \cite{GMPR26} 
constructed non-trivial weak solutions to Kdv equations with identically zero initial
data in 
$
\bigcap_{\varepsilon>0}
C_t^0L_x^{k-\varepsilon}([0,1]\times\mathbb T).
$
Gismondi, Golding and Novack \cite{GGN26} established a sharp rigidity--flexibility threshold for the isotropic Landau equation at $L^{6/5}(\mathbb{R}^3)$, 
They constructed non-trivial nonnegative stationary solutions in $L^1(\mathbb{R}^3)\cap L^p(\mathbb{R}^3)$ for every $1<p<6/5$.
Radu \cite{RaduKP26} extended flexibility methods to the
Kadomtsev--Petviashvili equations who constructed infinitely many 
singular weak solutions with zero initial data and compact support in time
for third- and fifth-order KP-I, KP-II
on $\mathbb T^2$ and $\mathbb R^2$. 

For nonlinear wave equations, few non-uniqueness result is known. 
For the quasi-linear systems, 
Mao and Qu \cite{MaoQu2025Lame,MaoQu2025Null} proved non-uniqueness
for the nonlinear dynamical Lam\'e system,  
and they constructed non-trivial $C^1$ weak solutions with zero initial data for a nonlinear
elastodynamic system satisfying the null condition. 
And in this paper, we aim to provide a non-uniqueness result for semilinear wave equations. 

\subsection{Main results.}
In this section, we state the main
non-uniqueness result for \eqref{eq:NLW}. 

First, we recall the basic tool of Littlewood--Paley decomposition.
For any distribution $u\in\mathscr{D}'(\mathbb{T})$, we denote by
$\widehat{u}$ its Fourier transform, which is defined by
\[
    \widehat{u}(k)
    :=
    \left\langle
        u,e^{-2\pi i k x}
    \right\rangle,
    \qquad k\in\mathbb{Z},
\]
where $\langle\cdot,\cdot\rangle$ denotes the duality pairing between
distributions and smooth periodic functions. Let $\chi\in C_c^\infty(\mathbb{R})$ be a even function supported in
$[-1,1]$ and satisfying
\[
    \chi\equiv1
    \qquad\text{ on } [-9/10,9/10].
\]
For each dyadic integer $K\geq1$, we define the Littlewood--Paley
projection $P_{\leq K}$ by
\begin{align}
    P_{\leq K}u(x)
    :=
    \sum_{k\in\mathbb{Z}}
    \chi(k/K)\widehat{u}(k)e^{2\pi i k x}.
    \label{Pkoper1}
\end{align}
So we can define
\begin{align}
    P_{>K}
    :=
    \operatorname{Id}-P_{\leq K},
    \quad
    P_1:=P_{\leq1},
    \quad
    P_K
    :=
    P_{\leq K}-P_{\leq K/2},
    \quad K\geq2.
    \label{Pkoper2}
\end{align}

Since the solutions considered in this paper have low regularity, we need to define the
nonlinear products in negative Sobolev spaces. 
Therefore, we recall the inhomogeneous Sobolev space
$\Hs$ equipped with the norm
\[
    \|u\|_{\Hs}
    :=
    \left(
        \sum_{k\in\mathbb{Z}}
        (1+k^2)^{-s}|\widehat{u}(k)|^2
    \right)^{1/2}.
\]
We define 
\begin{equation}
    \label{e:u*v}
    u \cdot v := \sum_{\tau_1, \tau_2} P_{\tau_1} u \cdot P_{\tau_2} v \quad \text{ in } C_t\Hs, 
\end{equation}
 where $u,v \in \mathscr{D}'(\mathbb{T})$, $ s>0$, $\tau_i\ (i=1,2)$ are dyadic integers and $\displaystyle\sum_{\tau_1, \tau_2}$ denotes summation among all dyadic integer pairs.
 If 
 \begin{equation}
    \label{e:u*u}
    \sum_{\tau_1, \tau_2}\| P_{\tau_1} u \cdot P_{\tau_2} u \| _{C_t\Hs }< \infty,
\end{equation}
we can say that $u \cdot u$ is well-defined as a product in $C_t\Hs$.
\begin{remark}
    In \cite{GGN26}, the nonlinear terms are defined by mollification which
    is equivalent to  \eqref{e:u*u} for the nonlinearity of system \eqref{eq:NLW} in this paper.
    Indeed, if $u_\varepsilon=u*\rho_\varepsilon$, 
    where $\rho_\varepsilon(x):= \sum_{k\in\mathbb{Z}}\chi(\varepsilon k) e^{2\pi i kx}$, 
    then we have 
\[
\partial_tu_\varepsilon=(\partial_tu)*\rho_\varepsilon,
\qquad
\nabla u_\varepsilon=(\nabla u)*\rho_\varepsilon.
\]
Therefore, the quadratic term
\[
\Theta_1(\partial_tu_\varepsilon)^2+\Theta_2|\nabla u_\varepsilon|^2
\]
is exactly the mollified nonlinear term in the definition \eqref{e:u*u}.
If its distributional limit exists as $\varepsilon\to0$, it is equivalent to \eqref{e:u*u}. 
\end{remark}
Now, we can define the singular weak solutions.

\begin{definition}[Singular weak solution]\label{def:weak-solution}
Let $T>0$, $0<\beta<1$, 
we say that
\[
u\in C_{t,x}^{\beta}([0,T]\times\mathbb T),
\qquad
\partial_tu\in C_t ([0,T], B_{\infty,\infty}^{\beta-1}(\mathbb T))
\]
is a singular weak solution to \eqref{eq:NLW} 
if the following conditions are satisfied.
\begin{itemize} 
    \item There exist $s>0$, such that $Q(u,u)$ is well-defined in $C_t([0,T];\Hs(\mathbb T))$, 
    where the bilinear form is defined by 
    \begin{align}
    Q(u,v):=\Theta_1\,\partial_tu\,\partial_tv + \Theta_2\,\partial_x u \partial_x v.
   \end{align}
    \item For every $\varphi\in C_c^\infty([0,T)\times\mathbb T)$, 
    \begin{align}\label{eq:weak-formulation}
        &\int_0^T\left\langle u(t),\,\partial_{tt}\varphi-\partial_{xx}\varphi\right\rangle\,dt-
        \int_0^T\left\langle Q(u,u),\varphi\right\rangle \,dt\notag\\
        &\qquad+\left\langle u_0,\partial_t\varphi(0)\right\rangle -\left\langle u_1,\varphi(0)\right\rangle =0.
    \end{align} 
     \item The initial conditions are satisfied in the sense that
     \begin{align}
        u(t)\longrightarrow u_0\quad\text{in }C_x^{\beta}(\mathbb T),\\
        \partial_tu(t)\longrightarrow u_1\quad\text{in }B_{\infty,\infty}^{\beta-1}(\mathbb T) 
     \end{align}
as $t\to0^+$.
\end{itemize}
\end{definition}
\begin{remark}
    We use the $H_x^{-s}$ instead of $\dot{H}_x^{-s}$ 
    which is different from the construction of singular weak solutions previously. 
    Thus, we can define the singular weak solution without divergent form and 
    meanwhile need different a priori estimates.
\end{remark}
Here, we give the main result of this paper Theorem \ref{thm1}, which states the existence of non-trivial singular weak solutions with zero initial data.
\begin{theorem}
    \label{thm1}
	Given $0< \beta<1,$  $1\le p_*<2$ and $\Theta_1\Theta_2<0$, 
	there exist infinitely many singular weak  solutions
	$$
	u\in
	C_{t,x}^{\beta}([0,T]\times\mathbb T) \cap C_t^1 ([0,T], B_{\infty,\infty}^{\beta-1}(\mathbb T))\cap C_{t}([0,T],  W^{1,p_*}(\mathbb T)) 
	$$
	with the initial data $u_{0} = u_1 = 0$ to the semilinear wave equations \eqref{eq:NLW}. 
\end{theorem}

\begin{remark}
    Due to the classical uniqueness result in $C_{t,x}^{1}([0,T]\times\mathbb T)$, see for instance
    Appendix \ref{app:auxiliary-lemma} of this paper, the non-uniqueness result 
    given in Theorem \ref{thm1} is sharp. 
\end{remark}

Now we would like to discuss the main difficulties caused by the structure of the semilinear wave equations 
and explain how to overcome them with new methods. 
Especially, we would like to explain the difference between fluid dynamics and semilinear wave equations. 

\noindent
(i) \textbf{Relaxed system.}
	For the classical Euler and Navier--Stokes equations, we usually write the relaxed system as 
\begin{equation}\label{eq:Euler-Reynolds}
\left\{
\begin{aligned}
&\partial_t v_q+\operatorname{div}(v_q\otimes v_q)+\nabla p_q -\mu \Delta u_q 
    =\operatorname{div}R_q,\notag\\
&\operatorname{div}v_q=0,
\end{aligned}
\right.
\end{equation}
where the Reynolds stress appears as the divergence form and would be cancelled by the nonlinearity.
However, semilinear equations \eqref{eq:NLW} possess a nonlinearity without the divergence form.  
So in order to cancel the error with the nonlinear terms, the error can not be in a derivative form, namely, 
we write the relaxed system as 
\begin{align}
    \partial_{tt}u_q-\partial_{xx} u_q - \big(\Theta_1(\partial_t u_q)^2 + \Theta_2(\partial_x u_q)^2 \big)=E_q.\notag
\end{align}
\noindent
(ii) \textbf{The error.}
    Tranditionally, convex integration methods often use the quadratic nonlinearity to
defect the error, in a manner, 
\begin{align}
\int_{\mathbb{T}^d} w_{q+1}\otimes w_{q+1}\,dx
\approx \varrho\,\mathrm{Id}-R_q. \notag
\end{align}
For the nonlinearity with divergent form, one usually deals 
with the correction term $\mathrm{div} \varrho\,\mathrm{Id} = \nabla\varrho$ which can then be absorbed into the pressure 
or just disappear when $\varrho$ is independent of $x$. 
However, this strategy fails here due to lack of the divergent form. 
To overcome this difficulty, we construct the perturbations composed of two parts with two propagation speeds. 
With the assumption $\Theta_1 \Theta_2 < 0$, we can write the low frequency of the nonlinearity as  
\begin{align}
\int_{\mathbb{T}} \Theta_1 (\partial_t w_{q+1})^2 + \Theta_2 (\partial_x w_{q+1})^2\,dx
\approx -E_q. \notag
\label{eq:scalar-defect-cancellation}
\end{align} 
Noting that we do not have an inverse-divergence
to define $E_q$. In order to utilize the benefits from the high frequency
of the perturbations, which is essential
to get the smallness of the error $E_q$, we need to estimate $E_q$ in
space of negative regularity, which is the main reason we use the singular
convex integration scheme in this paper. 
Moreover, we work in the nonhomogeneous space $H_x^{-s}$, which 
allows us to include the zero Fourier mode and define the weak solutions more conveniently.\\
\noindent
(iii) \textbf{The building blocks.} 
In order to cancel $E_q$ with low frequency part of
the nonlinearity, we need to carefully choose the two parts of the building blocks, especially their propagation
speeds \eqref{eq:three-compatible-conditions}, as well as a new form of
the amplitudes \eqref{eq:amplitude-plus-new} and \eqref{eq:amplitude-minus-new}. 
To keep the frequency of perturbations in a shell, we also introduce the correction part $w_{q+1}^{(c)}$ as in the previous singular convex integration schemes.  

\subsection{Organization.}
The paper is organized as follows.

In Section \ref{sec2}, we state the relaxed system and the main iteration lemma. In particular, we give the
inductive estimates that need to be proved on the approximate solutions and the associated
error terms. 

In Section \ref{sec3}, we are devoted to the construction of the perturbations. We first
introduce the intermittent building blocks and state their basic properties. 
Then we introduce $w_{q+1}^{(p)}$ as the principal part of the perturbation which would be 
used to cancel the error $E_q$.  
And we define  $w_{q+1}^{(c)}$ to keep the frequency of the perturbation in the shell.  
At the end of the section, we establish the estimates for the
perturbations in H\"older and Sobolev spaces.

In Section \ref{sec4}, we estimate the new error $E_{q+1}$. We decompose $E_{q+1}$ into three parts  
$E_{q+1}^{\mathrm{osc}}$, $E_{q+1}^{\mathrm{lin}}$ and $E_{q+1}^{\mathrm{cor}}$. 
and then give these estimates in Lemma \ref{lemma:hs}. 

In Section \ref{sec5}, we prove non-uniquenes result Theorem \ref{thm1} by applying the
iteration lemma and passing to the limits.

In Appendix \ref{app:auxiliary-lemma}, we show the uniqueness of solutions to \eqref{eq:NLW} in $C_{t,x}^1$ which follows the argument of Sideris \cite{Sideris1981Thesis}.  

\subsection{Notation.}
In this part, we explain for some main notations used in this paper:
\begin{enumerate}[label=\normalfont(\roman*)]
	\item  $X\lesssim Y$ means that $X\leq C Y$ for some constant $C>0$.
	\item  For $N\geq0$,
    $$
    \|f\|_{C^N_{t,x}}:=
    \sum_{m_1+m_2\leq N} \|\partial_t^{m_1}\partial_x^{m_2} f\|_{L^\infty_{t,x}}.
    $$
    And for $s>0$,
     $$
    \|f\|_{C_t\Hs}
    :=
    \sup_{t\in[0,T]}
    \left(\sum_{k\in\mathbb Z}(1+k^2)^{-s}|\widehat f(t,k)|^2
    \right)^{1/2}.
    $$
\end{enumerate}
\section{Inductive Lemma}\label{sec2}
We consider the relaxed system 
\begin{equation}\label{eq:approx-system}
    \left\{\aligned
	&   \partial_{tt}u_q-\partial_{xx} u_q - \big(\Theta_1(\partial_t u_q)^2 + \Theta_2(\partial_x u_q)^2 \big)=E_q,\\
    &   u_q(0,\cdot)=0,\qquad\partial_t u_q(0,\cdot)=0.
	\endaligned
	\right.
\end{equation}
where $u_q$ is the approximate solution to \eqref{eq:NLW} and $E_q$ is the error at the $q$-th stage with $q\in \mathbb{N}$.

Next, we take the frequency parameter $\lambda_q$ and the amplitude parameter $\delta_{q}$ which are chosen in the following way
\begin{equation}\label{la}
	\lambda_q=a^{(b^q)}, \ \
	\delta_{q}= \lambda_1^{3\gamma} \lambda_{q}^{-2\gamma}.
\end{equation}
Fix $0 < \varepsilon \ll s-1$ with $\varepsilon^{-1}\in \mathbb{N}$ depending on $\beta$ and $p_*$ in Theorem \ref{thm1}. 
Then the parameter $b\gg 1$ is chosen to satisfy 
\begin{align} \label{epsilon_condition}
b\varepsilon \in \mathbb{N}, 
\end{align}
where $b$ is a large integer.  
Moreover we select the appropriate parameter $\gamma$ that is a very small regularity parameter 
such that it satisfies $2 b\gamma \le \varepsilon$ and $\gamma b^2 \ll 1$.
Finally, we determine 
$a\in \mathbb{N}$ that is a large dyadic integer. 

Moreover, 
we introduce a decreasing sequence of temporal points $\{ t_q\}$ in order to make sure the
zero initial data for each $u_q$. For $T>0$, we define
\[
t_0:=\frac{T}{4},\qquad t_{q+1}:=t_q-\ell_q,
\qquad \ell_q:=\lambda_q^{-48}.
\]
For suitably large $a$, it is clear that  
\[
\sum_{q=0}^{\infty}\lambda_q^{-12}<\frac{t_0}{2},
\]
and thus we have 
\[
t_q\downarrow t_\infty:=t_0-\sum_{q=0}^{\infty}\lambda_q^{-12}>\frac{t_0}{2}>0.
\]
\begin{lemma}[Iteration lemma]
\label{lem:iteration}
Given $1\le p_* <2$ and $0 <\beta <1$, for any $s>1$, assume that $(u_q,E_q)$ is a smooth solution to \eqref{eq:approx-system}
satisfying 
\begin{align}
    &\|u_q\|_{C_{t,x}^N}\leq \lambda_{q}^{6N},\qquad 1\leq N\leq 8,\label{eq:uq-CN}\\
    &\|E_q\|_{C_t\Hs}\lesssim \delta_{q+1},\label{eq:inductive-defect}\\
    &\|E_q\|_{C_{t,x}^N} \lesssim \lambda_{q}^{12N+12}, \qquad 0\leq N\leq 6.\label{eq:inductive-defect-CN}
\end{align}
and 
\begin{align}
    &u_q=E_q=0 \qquad \text{on }[0,t_q]\times\mathbb T, \label{eq:time-support-q}\\
     &\supp_k {\hat{u}}_{q} \subset B(0,\lambda_{q}^5), \label{eq:space-support-q}
\end{align}
then there exists $(u_{q+1},E_{q+1})$ which is also a smooth solution to \eqref{eq:NLW} 
and satisfies 
 \begin{align}
    &\|u_{q+1}\|_{C_{t,x}^N}\leq \lambda_{q+1}^{6N},\qquad 1\leq N\leq 8,\label{eq1:uq-CN}\\
    &\|E_{q+1}\|_{C_t\Hs}\lesssim \delta_{q+2},\label{eq1:inductive-defect}\\
    &\|E_{q+1}\|_{C_{t,x}^N} \lesssim \lambda_{q+1}^{12N+12}, \qquad 0\leq N\leq 6.\label{eq1:inductive-defect-CN}  
\end{align}
and
\begin{align}
    &u_{q+1}=E_{q+1}=0 \qquad \text{on }[0,t_{q+1}]\times\mathbb T, \label{eq1:time-support-q}\\
    &\supp_k {\hat{u}}_{q+1} \subset B(0,\lambda_{q+1}^5), \label{eq1:space-support-q}
\end{align}
Moreover, the perturbation $w_{q+1} := u_{q+1} - u_q$ satisfies the following conditions:
\begin{enumerate}[label=\normalfont(\roman*)]
        \item 
        Frequency localization 
        \begin{align}\label{item:prop-hatw-supp}
        &\supp_k \widehat{w}_{q+1} \subset \left\{ k \in \mathbb{Z}: \frac{1}{2} \lambda_{q+1}^5 < |k| < \frac{9}{10}  \lambda_{q+1}^5 \right\}, 
        \end{align} 

        \item
        Small perturbation in low regularity spaces 
        \begin{align}\label{item:prop-w-Lp}
            \|w_{q+1}\|_{C_{t,x}^{\beta}} +\|\partial_t w_{q+1}\|_{C_tB_{\infty,\infty}^{\beta-1}} + \|w_{q+1}\|_{C_t W_x^{1,p_*}} < \delta_{q+2}; 
         \end{align} 

        \item
        Smallness of nonlinearity in $C_t\Hs$
\begin{align}\label{item:prop-w-paraproduct}
   & \|Q(w_{q+1},w_{q+1})\|_{C_t\Hs}\notag\\
    &+\sum_{M\leq 2N_q}
    \Big(|Q(P_Mu_q,w_{q+1})\|_{C_t\Hs} +\|Q(w_{q+1},P_Mu_q)\|_{C_t\Hs}\Big)
    \lesssim\delta_{q+1},
\end{align}
        where $M$ ranges over dyadic integers. 
    \end{enumerate}
\end{lemma}

\section{Perturbations}\label{sec3} 
In this section, our goal is to construct the perturbation $w_{q+1}$.
First, we introduce the principal parts $w^{(p)}_{q+1}$ which are characterized by four parameters $r$, $\lambda$, $\tau$ and $\sigma$:
\begin{equation}\label{choose-prar}
	r := \lambda_{q+1}^{-10\varepsilon},\
	\lambda := \lambda_{q+1},\   \tau:=\lambda_{q+1}^{2}, \ \tilde{\sigma}:=\lambda_{q+1}^{5}, \  \sigma : = \frac34 \tilde{\sigma} \in \mathbb{N}>50\tau^{2}.
\end{equation}
\subsection{Intermittency.}
As in \cite{bv19b}, we introduce spatial intermittent profiles. 
Let $\phi : \mathbb{R} \to \mathbb{R}$ be a smooth cut-off function supported on
the interval $[-1,1]$
and normalize $\phi$ to satisfy
\begin{equation}\label{e4.91}
	\int_{\mathbb{R}} \phi^2(x)\d x = 1,
\end{equation}
and 
\begin{equation}\label{e4.92}
  \int_{\mathbb{R}} \phi(x) \d x = 0.
\end{equation}
The corresponding rescaled cut-off functions are defined by
\begin{equation*}
	\phi_{r}(x) := {r^{-\frac{1}{2}}}\phi\left(\frac{x}{r}\right), 
\end{equation*}
Here,  $\phi_{r}$ is supported on the interval $[-r,r]$.
By an abuse of notation,
we periodize $\phi_{r}$ so that
they are treated as periodic functions defined on $\mathbb{T}$.

To simplify notations, we set
\begin{align}\label{snp-endpt2}
	\phi_{(r)}(x) := \phi_{r}(\lambda r x).
\end{align}
Then,
$\phi_{(r)}(x)$ is  mean zero on $\T$ with 
\begin{align} 
     \fint_{\mathbb{T}} \phi_{(r)}^2(x) \d x =1.
\end{align} 
\begin{lemma} [\cite{bbv20}\textit{Lemmas 5.1-5.2}] 
    \label{buildingblockestlemma}
	For any $p \in [1,\infty]$ and $N \in \mathbb{N}$, we have
	\begin{align}
		&\left\|\nabla^{N} \phi_{(r)}\right\|_{L^{p}}
		\lesssim_N {r}^{\frac 1p- \frac12}  \lambda^{N}.  \label{intermittent estimates2-endpt2}
	\end{align}
\end{lemma}

\subsection{Amplitudes of perturbations}
To ensure the cancellation of the error $E_q$ in the previous stage, 
while keeping the temporal support controlled, 
we choose the temporal cut-off function adapted to the time support of $E_q$. 
Let 
\[
    S_q:=\operatorname{supp}_t E_q
    =
    \overline{
    \left\{
    t\in[0,T]:
    E_q(t,\cdot)\not\equiv0
    \right\}}.
\]
and 
use the parameter $\ell_q$ as before, 
\begin{equation}\label{eq:epsilon-theta-choice}
    \ell_q : = \lambda_q^{-48}. \notag
\end{equation}
We choose a smooth temporal cut-off function
$\chi_q\in C^\infty([0,T])$ satisfying
\begin{align}
    &0\leq \chi_q\leq1,     \\
    &\chi_q\equiv1
    \quad\text{on }
    \mathcal N_{\ell_q/4}(S_q), \\
    &\operatorname{supp}_t\chi_q
    \subset
    \mathcal N_{\ell_q/2}(S_q),  \\
    &\|\partial_t^m\chi_q\|_{L_t^\infty}
    \lesssim \ell_q^{-m},
    \qquad 1\leq m\leq 8.\label{chi-est-time}
\end{align}
In particular,
\[
    \chi_q^2(t)E_q(t,x)=E_q(t,x)
    \qquad
    \text{for every }(t,x)\in[0,T]\times\mathbb T.
\]
Moreover,
\[
    \operatorname{supp}_t\partial_t^m\chi_q
    \subset
    \mathcal N_{\ell_q/2}(S_q)
    \setminus
    \mathcal N_{\ell_q/4}(S_q),
    \qquad \forall m\geq1,
\]
so that
\[
    E_q=0
    \qquad
    \text{on }
    \operatorname{supp}_t\partial_t^m\chi_q .
\]
By our basic assumption $\Theta_1 \Theta_2 < 0$, 
we assume without loss of generality that 
\begin{equation}\label{eq:sign-ab}
    \Theta_1>0,\qquad \Theta_2<0,
\end{equation}
while for the case $ \Theta_1<0$ and $\Theta_2>0$, one can just treat $\tilde{u}=-u$. 
Define
\begin{equation}\label{eq:rpm-choice}
    r_+:=\mu\left(-\frac{\Theta_2}{\Theta_1}\right)^{1/2},
    \qquad
    r_-:=\mu^{-1}\left(-\frac{\Theta_2}{\Theta_1}\right)^{1/2},
\end{equation}
which are positive constants for any given parameter $\mu>1$. 
For instance, we can just choose $\mu=2$.  
Moreover, their product satisfies 
\begin{align}\label{eq:cross-direction-cancellation}
    \Theta_1r_+r_-+\Theta_2
    &=
    \Theta_1\left[
        \mu\left(-\frac{\Theta_2}{\Theta_1}\right)^{1/2}
    \right]
    \left[
        \mu^{-1}\left(-\frac{\Theta_2}{\Theta_1}\right)^{1/2}
    \right]+\Theta_2 =0.
\end{align}
By the definition of $r_+$ in \eqref{eq:rpm-choice}, we have
\begin{align}\label{eq:rplus-sign}
    \Theta_1r_+^2+\Theta_2
    &=
    \Theta_1\mu^2\left(-\frac{\Theta_2}{\Theta_1}\right)+\Theta_2
    =
    -\Theta_2\mu^2+\Theta_2
    =
    \Theta_2(1-\mu^2)
    >
    0,
\end{align}
where the last inequality follows from $\Theta_2<0$ and $\mu>1$. On the other hand, the definition of $r_-$ gives
\begin{align}\label{eq:rminus-sign}
    \Theta_1r_-^2+\Theta_2
    &=
    \Theta_1\mu^{-2}\left(-\frac{\Theta_2}{\Theta_1}\right)+\Theta_2
    =
    -\Theta_2\mu^{-2}+\Theta_2
    =
    \Theta_2(1-\mu^{-2})
    <
    0.
\end{align}
Therefore, the choice \eqref{eq:rpm-choice} simultaneously guarantees
\begin{equation}\label{eq:three-compatible-conditions}
    \Theta_1r_+^2+\Theta_2>0,\qquad
    \Theta_1r_-^2+\Theta_2<0,\qquad
    \Theta_1r_+r_-+\Theta_2=0.
\end{equation}
These basic facts would be used in our analysis on the nonlinear interactions of the perturbations. 

Now we define the amplitudes $a_{(+)}(t,x)$ and $a_{(-)}(t,x)$ as 
\begin{align}
    a_{(+)}(t,x)
    &:=
    \chi_q(t)
    \left(
        \frac{\rho_{q+1}(t,x)+E_q(t,x)}
        {4\pi^2(\Theta_1r_+^2+\Theta_2)}
    \right)^{1/2}, \label{eq:amplitude-plus-new}\\
    a_{(-)}(t,x)
    &:=
    \chi_q(t)
    \left(
        \frac{\rho_{q+1}(t,x)-E_q(t,x)}
        {-4\pi^2(\Theta_1r_-^2+\Theta_2)}
    \right)^{1/2}. \label{eq:amplitude-minus-new}
\end{align}
where 
\begin{equation}\label{eq:rho-amplitude}
    \rho_{q+1}(t,x)
    :=
    \left(E_q(t,x)^2+1\right)^{1/2},
\end{equation}
Since $\rho_{q+1}>|E_q|$, we can get 
\[
    \rho_{q+1}-E_q>0,
    \qquad
    \rho_{q+1}+E_q>0
\]
on $[0,T]\times\mathbb T$. 
The amplitudes defined in \eqref{eq:amplitude-plus-new}--\eqref{eq:amplitude-minus-new} are smooth, since $\varepsilon_{q+1}>0$ guarantees that the arguments of both square roots remain strictly positive.
For the estimates of $a_{(+)}$ and $a_{(-)}$,
by \eqref{eq:inductive-defect-CN}, \eqref{eq:epsilon-theta-choice} and \eqref{chi-est-time}, we have  
\begin{equation}\label{eq:rho-E-composition}
    \left\|
    \bigl(\rho_{q+1}\pm E_q\bigr)^{1/2}
    \right\|_{C_{t,x}^N}
    \lesssim
    \ell_q^{-(N+1)},
    \qquad
    0\leq N\leq6.
\end{equation}
Then, we can get 
\begin{equation}\label{eq:amplitude-estimate}
    \|a_{(+)}\|_{C_{t,x}^N}
    +
    \|a_{(-)}\|_{C_{t,x}^N}
    \lesssim
    \ell_q^{-(N+1)},
    \qquad
   0\leq N\leq6.
\end{equation}

\subsection{Construction of the perturbations.}
Now, we define the principal perturbation as 
\begin{align}\label{eq:two-building-blocks}
    w_{(+)}^{(p)}(t,x)
    &:=
    \frac{1}{\sigma}a_{(+)}(t,x)
    \phi_{(r)}(x)
    \cos\left(2\pi\sigma(x+r_+t)\right), \\
    w_{(-)}^{(p)}(t,x)
    &:=
    \frac{1}{\sigma}a_{(-)}(t,x)
      \phi_{(r)}(x)
    \cos\left(2\pi\sigma(x+r_-t)\right),
\end{align}
and 
\begin{equation}\label{eq:total-perturbation-two-directions}
    w_{q+1}^{(p)}:=w_{(+)}^{(p)}+w_{(-)}^{(p)}.
\end{equation}
We first calculate the time derivatives of the two principal perturbations, 
\begin{align}
    \partial_t w_{(+)}^{(p)}
    &=
    \frac{1}{\sigma}
    \partial_t a_{(+)}(t,x)
    \phi_{(r)}(x)
    \cos\left(2\pi\sigma(x+r_+t)\right)\notag\\
    &\quad
    -
    2\pi r_+
    a_{(+)}(t,x)
    \phi_{(r)}(x)
    \sin\left(2\pi\sigma(x+r_+t)\right),\\
    \partial_t w_{(-)}^{(p)}
    &=
    \frac{1}{\sigma}
    \partial_t a_{(-)}(t,x)
    \phi_{(r)}(x)
    \cos\left(2\pi\sigma(x+r_-t)\right)\notag\\
    &\quad
    -
    2\pi r_-
    a_{(-)}(t,x)
    \phi_{(r)}(x)
    \sin\left(2\pi\sigma(x+r_-t)\right).
\end{align}
Similarly,
the spatial derivatives are given by
\begin{align}
    \partial_x w_{(+)}^{(p)}
    &=
    \frac{1}{\sigma}
    \partial_x a_{(+)}(t,x)
    \phi_{(r)}(x)
    \cos\left(2\pi\sigma(x+r_+t)\right)\notag\\
    &\quad
    +
    \frac{1}{\sigma}
    a_{(+)}(t,x)
    \partial_x\phi_{(r)}(x)
    \cos\left(2\pi\sigma(x+r_+t)\right)\notag\\
    &\quad
    -
    2\pi
    a_{(+)}(t,x)
    \phi_{(r)}(x)
    \sin\left(2\pi\sigma(x+r_+t)\right),\\
    \partial_x w_{(-)}^{(p)}
    &=
    \frac{1}{\sigma}
    \partial_x a_{(-)}(t,x)
    \phi_{(r)}(x)
    \cos\left(2\pi\sigma(x+r_-t)\right)\notag\\
    &\quad
    +
    \frac{1}{\sigma}
    a_{(-)}(t,x)
    \partial_x\phi_{(r)}(x)
    \cos\left(2\pi\sigma(x+r_-t)\right)\notag\\
    &\quad
    -
    2\pi
    a_{(-)}(t,x)
    \phi_{(r)}(x)
    \sin\left(2\pi\sigma(x+r_-t)\right).
\end{align}
Hence, 
\begin{align}
    \partial_t w_{q+1}^{(p)}
    &=
    -2\pi r_+
    a_{(+)}(t,x)
    \phi_{(r)}(x)
    \sin\left(2\pi\sigma(x+r_+t)\right)\notag\\
    &\quad
    -
    2\pi r_-
    a_{(-)}(t,x)
    \phi_{(r)}(x)
    \sin\left(2\pi\sigma(x+r_-t)\right)\notag\\
    &\quad
    +
    \frac{1}{\sigma}
    \partial_t a_{(+)}(t,x)
    \phi_{(r)}(x)
    \cos\left(2\pi\sigma(x+r_+t)\right)\notag\\
    &\quad
    +
    \frac{1}{\sigma}
    \partial_t a_{(-)}(t,x)
    \phi_{(r)}(x)
    \cos\left(2\pi\sigma(x+r_-t)\right),\\
    \partial_x w_{q+1}^{(p)}
    &=
    -2\pi
    a_{(+)}(t,x)
    \phi_{(r)}(x)
    \sin\left(2\pi\sigma(x+r_+t)\right)\notag\\
    &\quad
    -
    2\pi
    a_{(-)}(t,x)
    \phi_{(r)}(x)
    \sin\left(2\pi\sigma(x+r_-t)\right)\notag\\
    &\quad
    +
    \frac{1}{\sigma}
    \partial_x a_{(+)}(t,x)
    \phi_{(r)}(x)
    \cos\left(2\pi\sigma(x+r_+t)\right)\notag\\
    &\quad
    +
    \frac{1}{\sigma}
    a_{(+)}(t,x)
    \partial_x\phi_{(r)}(x)
    \cos\left(2\pi\sigma(x+r_+t)\right)\notag\\
    &\quad
    +
    \frac{1}{\sigma}
    \partial_x a_{(-)}(t,x)
    \phi_{(r)}(x)
    \cos\left(2\pi\sigma(x+r_-t)\right)\notag\\
    &\quad
    +
    \frac{1}{\sigma}
    a_{(-)}(t,x)
    \partial_x\phi_{(r)}(x)
    \cos\left(2\pi\sigma(x+r_-t)\right).
\end{align}
We now calculate the nonlinear interactions as 
\begin{align}
    &\Theta_1\left(\partial_t w_{q+1}^{(p)}\right)^2
    +\Theta_2\left(\partial_x w_{q+1}^{(p)}\right)^2
    + E_q(t,x) \notag\\
    &=
    \Theta_1\Bigg[
    -2\pi r_+a_{(+)}(t,x)\phi_{(r)}(x)
    \sin\left(2\pi\sigma(x+r_+t)\right) \notag\\
    &\qquad
    -2\pi r_-a_{(-)}(t,x)\phi_{(r)}(x)
    \sin\left(2\pi\sigma(x+r_-t)\right) \notag\\
    &\qquad
    +\frac{1}{\sigma}\partial_ta_{(+)}(t,x)\phi_{(r)}(x)
    \cos\left(2\pi\sigma(x+r_+t)\right) \notag\\
    &\qquad
    +\frac{1}{\sigma}\partial_ta_{(-)}(t,x)\phi_{(r)}(x)
    \cos\left(2\pi\sigma(x+r_-t)\right)
    \Bigg]^2 \notag\\
    &\quad
    +\Theta_2\Bigg[
    -2\pi a_{(+)}(t,x)\phi_{(r)}(x)
    \sin\left(2\pi\sigma(x+r_+t)\right) \notag\\
    &\qquad
    -2\pi a_{(-)}(t,x)\phi_{(r)}(x)
    \sin\left(2\pi\sigma(x+r_-t)\right) \notag\\
    &\qquad
    +\frac{1}{\sigma}
    \left(
    \partial_xa_{(+)}(t,x)\phi_{(r)}(x)
    +a_{(+)}(t,x)\partial_x\phi_{(r)}(x)
    \right)
    \cos\left(2\pi\sigma(x+r_+t)\right) \notag\\
    &\qquad
    +\frac{1}{\sigma}
    \left(
    \partial_xa_{(-)}(t,x)\phi_{(r)}(x)
    +a_{(-)}(t,x)\partial_x\phi_{(r)}(x)
    \right)
    \cos\left(2\pi\sigma(x+r_-t)\right)
    \Bigg]^2  + E_q(t,x)
    \notag\\
    &=
     E_q(t,x)+
    4\pi^2(\Theta_1r_+^2+\Theta_2)a_{(+)}^2(t,x)\phi_{(r)}^2(x)
    \sin^2\left(2\pi\sigma(x+r_+t)\right) \notag\\
    &\quad
    +4\pi^2(\Theta_1r_-^2+\Theta_2)a_{(-)}^2(t,x)\phi_{(r)}^2(x)
    \sin^2\left(2\pi\sigma(x+r_-t)\right) \notag\\
    &\quad
    +\underbrace{
    8\pi^2(\Theta_1r_+r_-+\Theta_2)a_{(+)}(t,x)a_{(-)}(t,x)\phi_{(r)}^2(x)
    \sin\left(2\pi\sigma(x+r_+t)\right)
    \sin\left(2\pi\sigma(x+r_-t)\right)
    }_{E^{\mathrm{osc},1}} \notag\\
    &\quad
    \underbrace{
    \begin{aligned}
        &-\frac{4\pi \Theta_1 r_+}{\sigma}
        a_{(+)}(t,x)\partial_ta_{(+)}(t,x)\phi_{(r)}^2(x)
        \sin\left(2\pi\sigma(x+r_+t)\right)
        \cos\left(2\pi\sigma(x+r_+t)\right)\\
        &-\frac{4\pi \Theta_1 r_+}{\sigma}
        a_{(+)}(t,x)\partial_ta_{(-)}(t,x)\phi_{(r)}^2(x)
        \sin\left(2\pi\sigma(x+r_+t)\right)
        \cos\left(2\pi\sigma(x+r_-t)\right)\\
        &-\frac{4\pi \Theta_1 r_-}{\sigma}
        a_{(-)}(t,x)\partial_ta_{(+)}(t,x)\phi_{(r)}^2(x)
        \sin\left(2\pi\sigma(x+r_-t)\right)
        \cos\left(2\pi\sigma(x+r_+t)\right)\\
        &-\frac{4\pi \Theta_1 r_-}{\sigma}
        a_{(-)}(t,x)\partial_ta_{(-)}(t,x)\phi_{(r)}^2(x)
        \sin\left(2\pi\sigma(x+r_-t)\right)
        \cos\left(2\pi\sigma(x+r_-t)\right)\\
        &-\frac{4\pi \Theta_2}{\sigma}
        a_{(+)}(t,x)\phi_{(r)}(x)
        \left(
        \partial_xa_{(+)}(t,x)\phi_{(r)}(x)
        +a_{(+)}(t,x)\partial_x\phi_{(r)}(x)
        \right)\\
        &\qquad\qquad
        \sin\left(2\pi\sigma(x+r_+t)\right)
        \cos\left(2\pi\sigma(x+r_+t)\right)\\
        &-\frac{4\pi \Theta_2}{\sigma}
        a_{(+)}(t,x)\phi_{(r)}(x)
        \left(
        \partial_xa_{(-)}(t,x)\phi_{(r)}(x)
        +a_{(-)}(t,x)\partial_x\phi_{(r)}(x)
        \right)\\
        &\qquad\qquad
        \sin\left(2\pi\sigma(x+r_+t)\right)
        \cos\left(2\pi\sigma(x+r_-t)\right)\\
        &-\frac{4\pi \Theta_2}{\sigma}
        a_{(-)}(t,x)\phi_{(r)}(x)
        \left(
        \partial_xa_{(+)}(t,x)\phi_{(r)}(x)
        +a_{(+)}(t,x)\partial_x\phi_{(r)}(x)
        \right)\\
        &\qquad\qquad
        \sin\left(2\pi\sigma(x+r_-t)\right)
        \cos\left(2\pi\sigma(x+r_+t)\right)\\
        &-\frac{4\pi \Theta_2}{\sigma}
        a_{(-)}(t,x)\phi_{(r)}(x)
        \left(
        \partial_xa_{(-)}(t,x)\phi_{(r)}(x)
        +a_{(-)}(t,x)\partial_x\phi_{(r)}(x)
        \right)\\
        &\qquad\qquad
        \sin\left(2\pi\sigma(x+r_-t)\right)
        \cos\left(2\pi\sigma(x+r_-t)\right)
    \end{aligned}
    }_{E^{\mathrm{osc},2}} \notag\\
    &\quad
    \underbrace{
    \begin{aligned}
        &+\frac{\Theta_1}{\sigma^2}
        \left[
        \partial_ta_{(+)}(t,x)\phi_{(r)}(x)
        \cos\left(2\pi\sigma(x+r_+t)\right)
        \right]^2\\
        &+\frac{\Theta_1}{\sigma^2}
        \left[
        \partial_ta_{(-)}(t,x)\phi_{(r)}(x)
        \cos\left(2\pi\sigma(x+r_-t)\right)
        \right]^2\\
        &+\frac{2\Theta_1}{\sigma^2}
        \partial_ta_{(+)}(t,x)\partial_ta_{(-)}(t,x)\phi_{(r)}^2(x)
        \cos\left(2\pi\sigma(x+r_+t)\right)
        \cos\left(2\pi\sigma(x+r_-t)\right)\\
        &+\frac{\Theta_2}{\sigma^2}
        \left[
        \left(
        \partial_xa_{(+)}(t,x)\phi_{(r)}(x)
        +a_{(+)}(t,x)\partial_x\phi_{(r)}(x)
        \right)
        \cos\left(2\pi\sigma(x+r_+t)\right)
        \right]^2\\
        &+\frac{\Theta_2}{\sigma^2}
        \left[
        \left(
        \partial_xa_{(-)}(t,x)\phi_{(r)}(x)
        +a_{(-)}(t,x)\partial_x\phi_{(r)}(x)
        \right)
        \cos\left(2\pi\sigma(x+r_-t)\right)
        \right]^2\\
        &+\frac{2\Theta_2}{\sigma^2}
        \left(
        \partial_xa_{(+)}(t,x)\phi_{(r)}(x)
        +a_{(+)}(t,x)\partial_x\phi_{(r)}(x)
        \right)\\
        &\qquad\qquad
        \left(
        \partial_xa_{(-)}(t,x)\phi_{(r)}(x)
        +a_{(-)}(t,x)\partial_x\phi_{(r)}(x)
        \right)\\
        &\qquad\qquad
        \cos\left(2\pi\sigma(x+r_+t)\right)
        \cos\left(2\pi\sigma(x+r_-t)\right).
    \end{aligned}
    }_{E^{\mathrm{osc},3}}\label{def:Eosc}
\end{align}
By the choices of $a_{(+)}(t,x)$ and $a_{(-)}(t,x)$, 
we can get 
\begin{align}
    &4\pi^2(\Theta_1r_+^2+\Theta_2)a_{(+)}^2(t,x)\phi_{(r)}^2(x)
    \sin^2\left(2\pi\sigma(x+r_+t)\right) \notag\\
    &\quad
    +4\pi^2(\Theta_1r_-^2+\Theta_2)a_{(-)}^2(t,x)\phi_{(r)}^2(x)
    \sin^2\left(2\pi\sigma(x+r_-t)\right)-E_q(t,x)\notag\\
    &=
    2\pi^2(\Theta_1r_+^2+\Theta_2)a_{(+)}^2(t,x)\phi_{(r)}^2(x) \notag\\
    &\quad
    +2\pi^2(\Theta_1r_-^2+\Theta_2)a_{(-)}^2(t,x)\phi_{(r)}^2(x) -E_q(t,x)\notag\\
    &\quad
    -2\pi^2(\Theta_1r_+^2+\Theta_2)a_{(+)}^2(t,x)\phi_{(r)}^2(x)
    \cos\left(4\pi\sigma(x+r_+t)\right) \notag\\
    &\quad
    -2\pi^2(\Theta_1r_-^2+\Theta_2)a_{(-)}^2(t,x)\phi_{(r)}^2(x)
    \cos\left(4\pi\sigma(x+r_-t)\right) \notag\\
    &=
    2\pi^2(\Theta_1r_+^2+\Theta_2)a_{(+)}^2(t,x)P_{>(\lambda r)/2}\left(\phi_{(r)}^2(x)\right) \notag\\
    &\quad
    +2\pi^2(\Theta_1r_-^2+\Theta_2)a_{(-)}^2(t,x)P_{>(\lambda r)/2}\left(\phi_{(r)}^2(x)\right)\notag\\
    &\quad
    -2\pi^2(\Theta_1r_+^2+\Theta_2)a_{(+)}^2(t,x)\phi_{(r)}^2(x)
    \cos\left(4\pi\sigma(x+r_+t)\right) \notag\\
    &\quad
    -2\pi^2(\Theta_1r_-^2+\Theta_2)a_{(-)}^2(t,x)\phi_{(r)}^2(x)
    \cos\left(4\pi\sigma(x+r_-t)\right).
\end{align}

To correct the frequency of $w_{q+1}^{(p)}$, we define  
\begin{align}
    w_{(+)}^{(c)}(t,x)
    &:=
    -\frac{1}{\sigma}
    P_{>\tau}\left(a_{(+)}\right)(t,x)
    \phi_{(r)}(x)
    \cos\left(2\pi\sigma(x+r_+t)\right)\notag\\
    &\quad
    -\frac{1}{\sigma}
    P_{\leq\tau}\left(a_{(+)}\right)(t,x)
    P_{>\tau}\left(\phi_{(r)}\right)(x)
    \cos\left(2\pi\sigma(x+r_+t)\right),\label{def:wc1}
\end{align}
\begin{align}
    w_{(-)}^{(c)}(t,x)
    &:=
    -\frac{1}{\sigma}
    P_{>\tau}\left(a_{(-)}\right)(t,x)
    \phi_{(r)}(x)
    \cos\left(2\pi\sigma(x+r_-t)\right)\notag\\
    &\quad
    -\frac{1}{\sigma}
    P_{\leq\tau}\left(a_{(-)}\right)(t,x)
    P_{>\tau}\left(\phi_{(r)}\right)(x)
    \cos\left(2\pi\sigma(x+r_-t)\right),\label{def:wc2}
\end{align}
and 
\begin{align}
    w_{q+1}^{(c)}
    &:=
    w_{(+)}^{(c)}
    +
    w_{(-)}^{(c)}.
\end{align}

We now verify the resulting frequency localization, 
by using
$a_{(+)}
=
P_{\leq\tau}(a_{(+)})
+
P_{>\tau}(a_{(+)})$,
\begin{align}
    w_{(+)}^{(p)}
    +
    w_{(+)}^{(c)}=
    \frac{1}{\sigma}
    P_{\leq\tau}\left(a_{(+)}\right)(t,x)
    P_{\leq\tau}\left(\phi_{(r)}\right)(x)
    \cos\left(2\pi\sigma(x+r_+t)\right).
\end{align}
In the same way,
we obtain
\begin{align}
    w_{(-)}^{(p)}
    +
    w_{(-)}^{(c)}
    =
    \frac{1}{\sigma}
    P_{\leq\tau}\left(a_{(-)}\right)(t,x)
    P_{\leq\tau}\left(\phi_{(r)}\right)(x)
    \cos\left(2\pi\sigma(x+r_-t)\right).
\end{align}
Therefore,
the total perturbation is given by
\begin{align}
    w_{q+1}
    &:=
    w_{q+1}^{(p)}
    +
    w_{q+1}^{(c)}\notag\\
    &=
    \frac{1}{\sigma}
    P_{\leq\tau}\left(a_{(+)}\right)(t,x)
    P_{\leq\tau}\left(\phi_{(r)}\right)(x)
    \cos\left(2\pi\sigma(x+r_+t)\right)\notag\\
    &\quad
    +
    \frac{1}{\sigma}
    P_{\leq\tau}\left(a_{(-)}\right)(t,x)
    P_{\leq\tau}\left(\phi_{(r)}\right)(x)
    \cos\left(2\pi\sigma(x+r_-t)\right).
    \label{eq:total-frequency-corrected-perturbation}
\end{align}
We note that 
\begin{lemma}[\cite{CH26}\textit{Lemma 6.2}]
    \label{lemma:arithmetic-sigma}
    For large enough $a$ and parameters $\lambda_{q+1}$, $\tau $, $\sigma$, $\tilde{\sigma}$ given as in \eqref{choose-prar}, we can get 
    \[(\sigma - 2\tau,\sigma + 2\tau) \cup (-\sigma- 2\tau , -\sigma + 2\tau)  \subset \left\{k \in \mathbb{Z}:\frac{1}{2} \tilde{\sigma} < |k| < \frac{9}{10} \tilde{\sigma} \right\}. \]
\end{lemma}
\noindent Now, we use it to analyze the support of $w_{q+1}$.
\begin{lemma}
    \label{lemma:supp}
    Given $\tau,\sigma \ge1$ as in \eqref{choose-prar}, the perturbations
    \begin{equation}
        w_{q+1} := w^{(p)}_{q+1} + w^{(c)}_{q+1} 
    \end{equation}
    satisfies 
    \begin{align}
        \label{e:supp-hatw}
        &\supp_k \widehat{w}_{q+1} \subset \left\{  k \in \mathbb{Z}:\frac{1}{2} \tilde{\sigma} < |k| < \frac{9}{10} \tilde{\sigma} \right\},\\
        \label{e:supp-hatu}
        &\supp_k {\hat{u}}_{q+1} \subset B(0,\lambda_{q+1}^5).
    \end{align}
\end{lemma}
\begin{proof}
By the definition of the Littlewood--Paley projection, we have 
\[
\operatorname{supp}_k
\widehat{P_{\leq\tau}(a_{(+)})}\cup 
\operatorname{supp}_k
\widehat{P_{\leq\tau}(a_{(-)})}
\subset B(0,\tau),
\qquad
\operatorname{supp}_k
\widehat{P_{\leq\tau}(\phi_{(r)})}
\subset B(0,\tau).
\]
Since 
\[
\operatorname{supp}_k
\mathcal{F}_x
\left[
\cos\left(2\pi\sigma(x+r_{+,-} t)\right)
\right]
\subset
\{\sigma,-\sigma\},
\]
we can get 
\begin{align*}
\operatorname{supp}_k\widehat{w}_{q+1}
\subset
B(\sigma,2\tau)
\cup
B(-\sigma,2\tau).
\end{align*}
By Lemma~\ref{lemma:arithmetic-sigma},
\[
B(\sigma,2\tau)
\cup
B(-\sigma,2\tau)
\subset
\left\{
k\in\mathbb Z:
\frac12\tilde{\sigma}<|k|<\frac9{10}\tilde{\sigma}
\right\}.
\]
Therefore,
\[
\operatorname{supp}_k\widehat{w}_{q+1}
\subset
\left\{
k\in\mathbb Z:
\frac12\tilde{\sigma}<|k|<\frac9{10}\tilde{\sigma}
\right\},
\]
which proves \eqref{e:supp-hatw}. 
Since $\supp_k {\hat{u}}_{q} \subset B(0,\lambda_{q}^5)$ and $u_{q+1} = u_q + w_{q+1} $, 
we can get \eqref{e:supp-hatu} directly.
\end{proof}

\subsection{Estimates on the perturbations}
To estimate the perturbations, we first introduce a basic lemma. 
\begin{lemma}[\cite{CH26}\textit{Lemma 7.1}]\label{lemma:frequencylemma}
       Let $1 \le p \le \infty$. For any $f \in C^\infty(\mathbb{T}^d)$ and any dyadic integer $N$, it holds that
    \begin{align*}
        \|P_{\le N} f\|_{L^p} &\lesssim \|f\|_{L^p}, \\
        \|P_{>N} f\|_{L^p} &\lesssim N^{-1} \|\nabla f\|_{L^p}.
    \end{align*}
\end{lemma}
With this inequality, we can deduce the estimates for $w_{q+1}$. 
\begin{lemma}[Estimates on $w_{q+1}$]\label{lem:estimates-wq1}
For $1\leq p\leq\infty$,
we have the following estimates,
\begin{align}
    &\|\partial_{t}w_{q+1}^{(p)}\|_{C_tL_x^p} + \|\partial_{x}w_{q+1}^{(p)}\|_{C_tL_x^p}
    \lesssim
    \ell_q^{-1}r^{\frac1p-\frac12}, \label{eq:wp-derivative-new}\\
    &\|\partial_{t}w_{q+1}^{(c)}\|_{C_tL_x^p} + \|\partial_{x}w_{q+1}^{(c)}\|_{C_tL_x^p}
    \lesssim
    \ell_q^{-1}\tau^{-1}\lambda r^{\frac1p-\frac12}, \label{eq:wc-derivative-new}\\
    &\|w_{q+1}\|_{C_{t,x}^N}
    \lesssim
    \ell_q^{-(N+1)}\sigma^{-1}r^{-\frac12}
    +
    \ell_q^{-1}\sigma^{N-1}r^{-\frac12},
    \qquad 1\leq N\leq8, \label{eq:w-CN-new}\\
    &\|w_{q+1}\|_{C_{t,x}^{\beta}}
    +
    \|\partial_t w_{q+1}\|_{C_tB_{\infty,\infty}^{\beta-1}}
    \lesssim
    \ell_q^{-1}\sigma^{\beta-1}r^{-\frac12},
    \qquad 0\leq\beta<1.\label{eq:w-Cb-new}
\end{align}
\end{lemma}

\begin{proof}
By \eqref{eq:amplitude-estimate} and \eqref{intermittent estimates2-endpt2},
we estimate the derivatives of the principal perturbation,
where
\begin{align}
    \partial_t w_{(+)}^{(p)}
    &=
    -2\pi r_+
    a_{(+)}(t,x)\phi_{(r)}(x)
    \sin\left(2\pi\sigma(x+r_+t)\right) \notag\\
    &\quad
    +
    \frac{1}{\sigma}
    \partial_ta_{(+)}(t,x)\phi_{(r)}(x)
    \cos\left(2\pi\sigma(x+r_+t)\right),\\
    \partial_t w_{(-)}^{(p)}
    &=
    -2\pi r_-
    a_{(-)}(t,x)\phi_{(r)}(x)
    \sin\left(2\pi\sigma(x+r_-t)\right) \notag\\
    &\quad
    +
    \frac{1}{\sigma}
    \partial_ta_{(-)}(t,x)\phi_{(r)}(x)
    \cos\left(2\pi\sigma(x+r_-t)\right).
\end{align}
Hence, 
we obtain
\begin{align}
    \|\partial_t w_{q+1}^{(p)}\|_{C_tL_x^p}
    &\lesssim
    \left(
    \ell_q^{-1}
    +
    \sigma^{-1}\ell_q^{-2}
    \right)
    r^{\frac1p-\frac12}. \label{eq:dt-wp-pre}
\end{align}

For the spatial derivative,
we have
\begin{align}
    \partial_x w_{(+)}^{(p)}
    &=
    -2\pi
    a_{(+)}(t,x)\phi_{(r)}(x)
    \sin\left(2\pi\sigma(x+r_+t)\right) \notag\\
    &\quad
    +
    \frac{1}{\sigma}
    \partial_xa_{(+)}(t,x)\phi_{(r)}(x)
    \cos\left(2\pi\sigma(x+r_+t)\right) \notag\\
    &\quad
    +
    \frac{1}{\sigma}
    a_{(+)}(t,x)\partial_x\phi_{(r)}(x)
    \cos\left(2\pi\sigma(x+r_+t)\right),
\end{align}
Similarly we obtain
\begin{align}
    \|\partial_x w_{q+1}^{(p)}\|_{C_tL_x^p}
    &\lesssim
    \ell_q^{-1}r^{\frac1p-\frac12}
    +
    \sigma^{-1}\ell_q^{-2}r^{\frac1p-\frac12}\notag\\
    &\quad
    +
    \sigma^{-1}\ell_q^{-1}
    \lambda
    r^{\frac1p-\frac12}\notag\\
    &\lesssim
    \ell_q^{-1}r^{\frac1p-\frac12}. \label{eq:dx-wp-proof-new}
\end{align}
Combining \eqref{eq:dt-wp-pre} and \eqref{eq:dx-wp-proof-new},
we conclude that
\begin{align}
    \|\partial_{t}w_{q+1}^{(p)}\|_{C_tL_x^p}
    + \|\partial_{x}w_{q+1}^{(p)}\|_{C_tL_x^p}
    &\lesssim
    \ell_q^{-1}r^{\frac1p-\frac12}. \label{eq:wp-derivative-proof-new}
\end{align}

Next, we consider the frequency corrector defined by \eqref{def:wc1}--\eqref{def:wc2}, 
\begin{align}
    w_{q+1}^{(c)}
    &=
    -\frac{1}{\sigma}
    P_{>\tau}\left(a_{(+)}\right)(t,x)
    \phi_{(r)}(x)
    \cos\left(2\pi\sigma(x+r_+t)\right) \notag\\
    &\quad
    -\frac{1}{\sigma}
    P_{\leq\tau}\left(a_{(+)}\right)(t,x)
    P_{>\tau}\left(\phi_{(r)}\right)(x)
    \cos\left(2\pi\sigma(x+r_+t)\right) \notag\\
    &\quad
    -\frac{1}{\sigma}
    P_{>\tau}\left(a_{(-)}\right)(t,x)
    \phi_{(r)}(x)
    \cos\left(2\pi\sigma(x+r_-t)\right) \notag\\
    &\quad
    -\frac{1}{\sigma}
    P_{\leq\tau}\left(a_{(-)}\right)(t,x)
    P_{>\tau}\left(\phi_{(r)}\right)(x)
    \cos\left(2\pi\sigma(x+r_-t)\right). \label{eq:wc-new-proof}
\end{align}
By Lemma \ref{lemma:frequencylemma},
we have
\begin{align}
    &\left\|\partial_{t}w_{q+1}^{(c)}\right\|_{C_tL_x^p}
    +
    \left\|\partial_{x}w_{q+1}^{(c)}\right\|_{C_tL_x^p}\notag\\
    &\lesssim
    \left(
    \left\|P_{>\tau}a_{(+)}\right\|_{C_tL_x^\infty}
    +
    \left\|P_{>\tau}a_{(-)}\right\|_{C_tL_x^\infty}
    \right)
    \left\|\phi_{(r)}\right\|_{L_x^p} \notag\\
    &\quad
    +
    \sigma^{-1}
    \left(
    \left\|\partial_{t,x}P_{>\tau}a_{(+)}\right\|_{C_tL_x^\infty}
    +
    \left\|\partial_{t,x}P_{>\tau}a_{(-)}\right\|_{C_tL_x^\infty}
    \right)
    \left\|\phi_{(r)}\right\|_{L_x^p} \notag\\
    &\quad
    +
    \sigma^{-1}
    \left(
    \left\|P_{>\tau}a_{(+)}\right\|_{C_tL_x^\infty}
    +
    \left\|P_{>\tau}a_{(-)}\right\|_{C_tL_x^\infty}
    \right)
    \left\|\partial_x\phi_{(r)}\right\|_{L_x^p} \notag\\
    &\quad
    +
    \left(
    \left\|P_{\leq\tau}a_{(+)}\right\|_{C_tL_x^\infty}
    +
    \left\|P_{\leq\tau}a_{(-)}\right\|_{C_tL_x^\infty}
    \right)
    \left\|P_{>\tau}\phi_{(r)}\right\|_{L_x^p} \notag\\
    &\quad
    +
    \sigma^{-1}
    \left(
    \left\|\partial_{t,x}P_{\leq\tau}a_{(+)}\right\|_{C_tL_x^\infty}
    +
    \left\|\partial_{t,x}P_{\leq\tau}a_{(-)}\right\|_{C_tL_x^\infty}
    \right)
    \left\|P_{>\tau}\phi_{(r)}\right\|_{L_x^p} \notag\\
    &\quad
    +
    \sigma^{-1}
    \left(
    \left\|P_{\leq\tau}a_{(+)}\right\|_{C_tL_x^\infty}
    +
    \left\|P_{\leq\tau}a_{(-)}\right\|_{C_tL_x^\infty}
    \right)
    \left\|\partial_xP_{>\tau}\phi_{(r)}\right\|_{L_x^p} \notag\\
    &\lesssim
    \tau^{-1}\ell_q^{-2}r^{\frac1p-\frac12}
    +
    \sigma^{-1}\tau^{-1}\ell_q^{-3}r^{\frac1p-\frac12}
    +
    \sigma^{-1}\tau^{-1}\ell_q^{-2}
    \lambda r^{\frac1p-\frac12} \notag\\
    &\quad
    +
    \ell_q^{-1}\tau^{-1}\lambda r^{\frac1p-\frac12}
    +
    \sigma^{-1}\ell_q^{-2}\tau^{-1}\lambda r^{\frac1p-\frac12}
    +
    \sigma^{-1}\ell_q^{-1} \lambda \tau^{-1} r^{\frac1p-\frac12}
    \notag\\
    &\lesssim
    \ell_q^{-1}\tau^{-1}\lambda r^{\frac1p-\frac12}.
\end{align}

Finally,
by using the Leibniz rule, 
we obtain 
\begin{align}
    \|w_{q+1}\|_{C_{t,x}^N}
    &\lesssim
    \frac{1}{\sigma}
    \sum_{j+k+m=N}
    \sigma^{m}
    \left(
    \|a_{(+)}\|_{C_{t,x}^j}
    +
    \|a_{(-)}\|_{C_{t,x}^j}
    \right)
    \lambda^k
    r^{-\frac12}. \notag \\ 
    &\lesssim
    \frac{1}{\sigma}
    \sum_{j+k+m=N}
    \sigma^{N-j}
    \lambda^k
    \ell_q^{-(j+1)}
    r^{-\frac12}\notag \\
    &=
    \ell_q^{-1}\sigma^{N-1}r^{-\frac12},
    \qquad 1\leq N\leq8. \label{eq:CN-sharp-proof}
\end{align}
Moreover, we can calculate as 
\begin{align}
\|\partial_t w_{q+1}\|_{C_tB_{\infty,\infty}^{-1}}
&\lesssim
\sigma^{-1}
\left(
\|a_{(+)}\|_{C_{t,x}}
+
\|a_{(-)}\|_{C_{t,x}}
\right)
\|\phi_{(r)}\|_{L^\infty}\notag\\
&\quad +\sigma^{-2}
\left(
\|a_{(+)}\|_{C_{t,x}^1}
+
\|a_{(-)}\|_{C_{t,x}^1}
\right)
\|\phi_{(r)}\|_{L^\infty}
\notag\\
&\lesssim
\ell_q^{-1}\sigma^{-1}r^{-\frac12}+\ell_q^{-2}\sigma^{-2}r^{-\frac12}.
\end{align}
For $0\leq \beta<1$, by interpolation we can get 
\begin{align}
    \|w_{q+1}\|_{C_{t,x}^{\beta}}
    +
    \|\partial_t w_{q+1}\|_{C_tB_{\infty,\infty}^{\beta-1}}
    \lesssim
    \ell_q^{-1}\sigma^{\beta-1}r^{-\frac12}.
\end{align}
This completes the proof.
\end{proof}
\section{Estimates for the new error}\label{sec4}
\subsection{Decomposition of $E_{q+1}$}

Substituting $u_{q+1}=u_q+w_{q+1}$ into the approximate equation,
we obtain
\begin{align}
    E_{q+1}
    &=
    E_q
    +(\partial_{tt}-\partial_{xx})w_{q+1}
    -2\Theta_1\,\partial_tu_q\,\partial_tw_{q+1}
    -2\Theta_2\,\partial_x u_q\,\partial_x w_{q+1} \notag\\
    &\quad
    -\Theta_1\left(\partial_tw_{q+1}\right)^2
    -\Theta_2\left(\partial_x w_{q+1}\right)^2.
    \label{eq:E-next-new}
\end{align}
Since $w_{q+1}=w_{q+1}^{(p)}+w_{q+1}^{(c)}$,
we have
\begin{align}
    E_{q+1}
    &=
    \underbrace{
    E_q
    -\Theta_1\left(\partial_tw_{q+1}^{(p)}\right)^2
    -\Theta_2\left(\partial_x w_{q+1}^{(p)}\right)^2
    }_{E_{q+1}^{\mathrm{osc}}} \notag\\
    &\quad
    \underbrace{
    -2\Theta_1\,\partial_tw_{q+1}^{(p)}\partial_tw_{q+1}^{(c)}
    -\Theta_1\left(\partial_tw_{q+1}^{(c)}\right)^2
    -2\Theta_2\,\partial_x w_{q+1}^{(p)}\cdot\partial_x w_{q+1}^{(c)}
    -\Theta_2\left(\partial_x w_{q+1}^{(c)}\right)^2
    }_{E_{q+1}^{\mathrm{cor}}} \notag\\
    &\quad
    \underbrace{
    (\partial_{tt}-\partial_{xx})w_{q+1}
    -2\Theta_1\,\partial_tu_q\,\partial_tw_{q+1}
    -2\Theta_2\,\partial_x u_q\cdot\partial_x w_{q+1}
    }_{E_{q+1}^{\mathrm{lin}}}.
    \label{eq:E-decomposition-new}
\end{align}
Therefore,
we define
\begin{align}
    E_{q+1}
    =
    E_{q+1}^{\mathrm{osc}}
    +
    E_{q+1}^{\mathrm{cor}}
    +
    E_{q+1}^{\mathrm{lin}}.
\end{align}
\subsection{$C_t{H}_x^{-s}$-Estimates for the Error}

\begin{lemma}
\label{lemma:hs}
    For $s > 1$, we have the following estimates
    \begin{align}
        \label{esti:Eosc}
        &\|{E}_{q+1}^{\mathrm{osc}}\|_{C_t\Hs} \lesssim  \ell_q^{-3}(\lambda r)^{-1}r^{-\frac12} +  \ell_q^{-2}\sigma^{-s}r^{-\frac12}+\tau^{-1}\ell_q^{-2}\lambda, \\
        \label{esti:Elin}
        &\|{E}_{q+1}^{\mathrm{lin}}\|_{C_t\Hs} \lesssim  \sigma^{1-s}\ell_q^{-1} ,\\
         \label{esti:Ecor}
        &\|{E}_{q+1}^{\mathrm{cor}}\|_{C_t\Hs} \lesssim   \tau^{-1}\lambda\ell_q^{-3}.
    \end{align}
\end{lemma}

\begin{proof}
First, we calculate the $E_{q+1}^{\mathrm{osc}}$, 
\begin{align}
    E_{q+1}^{\mathrm{osc}}= - E^{\mathrm{osc},1}- E^{\mathrm{osc},2} - E^{\mathrm{osc},3}, \label{eq:osc error}
\end{align}
where $E^{\mathrm{osc},1}$, $E^{\mathrm{osc},2}$, $E^{\mathrm{osc},3}$ are defined in \eqref{def:Eosc}. 
For the first two terms of $E^{\mathrm{osc},1}$, 
by \eqref{eq:amplitude-estimate} 
we obtain
\begin{align}
    &\left\|
    2\pi^2(\Theta_1r_+^2+\Theta_2)a_{(+)}^2
    P_{>(\lambda r)/2}\left(\phi_{(r)}^2\right)
    \right\|_{C_tH_x^{-s}} \notag\\
    &\quad
    +
    \left\|
    2\pi^2(\Theta_1r_-^2+\Theta_2)a_{(-)}^2
    P_{>(\lambda r)/2}\left(\phi_{(r)}^2\right)
    \right\|_{C_tH_x^{-s}} \notag\\
    &\lesssim
    \left(
    \|a_{(+)}^2\|_{C_{t,x}^1}
    +
    \|a_{(-)}^2\|_{C_{t,x}^1}
    \right)
    (\lambda r)^{-1}
    \|\phi_{(r)}^2\|_{L_x^2} \notag\\
    &\lesssim
    \ell_q^{-3}(\lambda r)^{-1}
    \|\phi_{(r)}\|_{L_x^4}^2 \notag\\
    &\lesssim
    \ell_q^{-3}(\lambda r)^{-1}r^{-\frac12}.
    \label{eq:osc-preparation-new}
\end{align}
Next, we give thes estimate as
\begin{align}
    &\left\|
    2\pi^2(\Theta_1r_+^2+\Theta_2)a_{(+)}^2(t,x)\phi_{(r)}^2(x)
    \cos\left(4\pi\sigma(x+r_+t)\right)
    \right\|_{C_tH_x^{-s}} \notag\\
    &\lesssim
    \left\|
    2\pi^2(\Theta_1r_+^2+\Theta_2)
    P_{\leq\tau}\left(a_{(+)}^2\right)
    P_{\leq\tau}\left(\phi_{(r)}^2\right)
    \cos\left(4\pi\sigma(x+r_+t)\right)
    \right\|_{C_tH_x^{-s}} \notag\\
    &\quad
    +
    \left\|
    2\pi^2(\Theta_1r_+^2+\Theta_2)
    P_{>\tau}\left(a_{(+)}^2\right)
    P_{\leq\tau}\left(\phi_{(r)}^2\right)
    \cos\left(4\pi\sigma(x+r_+t)\right)
    \right\|_{C_tH_x^{-s}} \notag\\
    &\quad
    +
    \left\|
    2\pi^2(\Theta_1r_+^2+\Theta_2)a_{(+)}^2
    P_{>\tau}\left(\phi_{(r)}^2\right)
    \cos\left(4\pi\sigma(x+r_+t)\right)
    \right\|_{C_tH_x^{-s}} \notag\\
    &:=O_{+,1}+O_{+,2}+O_{+,3}.
    \label{eq:osc-plus-frequency-decomposition}
\end{align}
For the first term,
noticing the corresponding frequency locates in a high frequency annulus, we can have 
\begin{align}
    O_{+,1}
    &\lesssim
    \sigma^{-s}
    \left\|P_{\leq\tau}\left(a_{(+)}^2\right)\right\|_{C_tL_x^\infty}
    \left\|P_{\leq\tau}\left(\phi_{(r)}^2\right)\right\|_{L_x^2}\notag\\
    &\lesssim
    \sigma^{-s}
    \|a_{(+)}\|_{C_{t,x}}^2
    \|\phi_{(r)}\|_{L_x^4}^2\notag\\
    &\lesssim
    \ell_q^{-2}\sigma^{-s}r^{-\frac12}.
    \label{eq:O-plus-1}
\end{align}
For the second term, 
we obtain by using Lemma \ref{lemma:frequencylemma}
\begin{align}
    O_{+,2}
    &\lesssim
    \left\|P_{>\tau}\left(a_{(+)}^2\right)\right\|_{C_tL_x^\infty}
    \left\|\phi_{(r)}^2\right\|_{L_x^1}\notag\\
    &\lesssim
    \tau^{-1}
    \left\|\partial_x\left(a_{(+)}^2\right)\right\|_{C_{t,x}}
    \|\phi_{(r)}\|_{L_x^2}^2\notag\\
    &\lesssim
    \tau^{-1}\ell_q^{-3}.
    \label{eq:O-plus-2}
\end{align}
Similarly, for the third term we have
\begin{align}
    O_{+,3}
    &\lesssim
    \|a_{(+)}\|_{C_{t,x}}^2
    \left\|P_{>\tau}\left(\phi_{(r)}^2\right)\right\|_{L_x^1}\notag\\
    &\lesssim
    \tau^{-1}\ell_q^{-2}
    \left\|\partial_x\left(\phi_{(r)}^2\right)\right\|_{L_x^1}\notag\\
    &\lesssim
    \tau^{-1}\ell_q^{-2}
    \|\phi_{(r)}\|_{L_x^2}
    \|\partial_x\phi_{(r)}\|_{L_x^2}\notag\\
    &\lesssim
    \tau^{-1}\ell_q^{-2}\lambda.
    \label{eq:O-plus-3}
\end{align}
Combining the estimates of $ O_{+,1}$, $ O_{+,2}$ and $ O_{+,3}$, 
we obtain
\begin{align}
    &\left\|
    2\pi^2(\Theta_1r_+^2+\Theta_2)a_{(+)}^2(t,x)\phi_{(r)}^2(x)
    \cos\left(4\pi\sigma(x+r_+t)\right)
    \right\|_{C_tH_x^{-s}} \notag\\
    &\qquad\lesssim
    \ell_q^{-2}\sigma^{-s}r^{-\frac12}
    +
    \tau^{-1}\ell_q^{-3}
    +
    \tau^{-1}\ell_q^{-2}\lambda.
    \label{eq:osc-plus-final}
\end{align}
Similarly,
we can get 
\begin{align}
    &\left\|
    2\pi^2(\Theta_1r_-^2+\Theta_2)a_{(-)}^2(t,x)\phi_{(r)}^2(x)
    \cos\left(4\pi\sigma(x+r_-t)\right)
    \right\|_{C_tH_x^{-s}} \notag\\
    &\qquad\lesssim
    \ell_q^{-2}\sigma^{-s}r^{-\frac12}
    +
    \tau^{-1}\ell_q^{-3}
    +
    \tau^{-1}\ell_q^{-2}\lambda.
    \label{eq:osc-minus-final}
\end{align}
Since $-s < -1/2$, using Sobolev's embedding theorem, 
we know that $L_x^1$ embeds into $\Hs$.
So we get the estimate $E_{\mathrm{osc},2}$ as 
\begin{align}
    \|E_{\mathrm{osc},2}\|_{C_tL_x^{1}}
    &\lesssim
    \sigma^{-1}
    \left(
    \|a_{(+)}\|_{C_{t,x}}
    +
    \|a_{(-)}\|_{C_{t,x}}
    \right)
    \left(
    \|\partial_ta_{(+)}\|_{C_{t,x}}
    +
    \|\partial_ta_{(-)}\|_{C_{t,x}}
    \right)
    \|\phi_{(r)}\|_{L_x^2}^2 \notag\\
    &\quad
    +
    \sigma^{-1}
    \left(
    \|a_{(+)}\|_{C_{t,x}}
    +
    \|a_{(-)}\|_{C_{t,x}}
    \right)
    \left(
    \|\partial_xa_{(+)}\|_{C_{t,x}}
    +
    \|\partial_xa_{(-)}\|_{C_{t,x}}
    \right)
    \|\phi_{(r)}\|_{L_x^2}^2 \notag\\
    &\quad
    +
    \sigma^{-1}
    \left(
    \|a_{(+)}\|_{C_{t,x}}
    +
    \|a_{(-)}\|_{C_{t,x}}
    \right)^2
    \|\phi_{(r)}\|_{L_x^2}
    \|\partial_x\phi_{(r)}\|_{L_x^2} \notag\\
    &\lesssim
    \sigma^{-1}
    \left(
    \ell_q^{-3}
    +
    \ell_q^{-2}\lambda
    \right) \notag\\
    &\lesssim
    \sigma^{-1}\lambda \ell_q^{-3}.
    \label{eq:Eosc2-estimate}
\end{align}
Finally,
for $E_{\mathrm{osc},3}$,
we similarly obtain
\begin{align}
    \|E_{\mathrm{osc},3}\|_{C_tL_x^{1}}
    &\lesssim
    \sigma^{-2}
    \left(
    \|\partial_ta_{(+)}\|_{C_{t,x}}
    +
    \|\partial_ta_{(-)}\|_{C_{t,x}}
    \right)^2
    \|\phi_{(r)}\|_{L_x^2}^2 \notag\\
    &
    +
    \sigma^{-2}
    \left(
    \|\partial_xa_{(+)}\|_{C_{t,x}}
    +
    \|\partial_xa_{(-)}\|_{C_{t,x}}
    \right)^2
    \|\phi_{(r)}\|_{L_x^2}^2 \notag\\
    &
    +
    \sigma^{-2}
    \left(
    \|a_{(+)}\|_{C_{t,x}}
    +
    \|a_{(-)}\|_{C_{t,x}}
    \right)
    \left(
    \|\partial_xa_{(+)}\|_{C_{t,x}}
    +
    \|\partial_xa_{(-)}\|_{C_{t,x}}
    \right)\notag\\
    &\qquad\|\phi_{(r)}\|_{L_x^2}
    \|\partial_x\phi_{(r)}\|_{L_x^2} \notag\\
    &
    +
    \sigma^{-2}
    \left(
    \|a_{(+)}\|_{C_{t,x}}
    +
    \|a_{(-)}\|_{C_{t,x}}
    \right)^2
    \|\partial_x\phi_{(r)}\|_{L_x^2}^2 \notag\\
    &\lesssim
    \sigma^{-2}
    \left(
    \ell_q^{-4}
    +
    \ell_q^{-3}\lambda
    +
    \ell_q^{-2}\lambda^2
    \right) \notag\\
    &\lesssim
    \sigma^{-2}\ell_q^{-2}\lambda^2.
    \label{eq:Eosc3-estimate}
\end{align}
Combining \eqref{eq:osc-preparation-new}, \eqref{eq:osc-plus-final}--\eqref{eq:Eosc3-estimate}, 
we conclude that
\begin{align}
    \|E_{q+1}^{\mathrm{osc}}\|_{C_tH_x^{-s}}
    &\lesssim
    \ell_q^{-3} (\lambda r)^{-1}r^{-\frac12}
    +
    \ell_q^{-2}\sigma^{-s}r^{-\frac12}
    +
    \tau^{-1}\ell_q^{-2}\lambda.
    \label{eq:osc-final-estimate}
\end{align}
Next, we calculate the linear error as
\begin{align}
    &(\partial_{tt}-\partial_{xx})w_{q+1} \notag\\
    &=
    \frac{1}{\sigma}
    \partial_{tt}\left(P_{\leq\tau}(a_{(+)})\right)(t,x)
    P_{\leq\tau}(\phi_{(r)})(x)
    \cos\left(2\pi\sigma(x+r_+t)\right) \notag\\
    &\quad
    -\frac{1}{\sigma}
    \partial_{xx}\left(
    P_{\leq\tau}(a_{(+)})(t,x)
    P_{\leq\tau}(\phi_{(r)})(x)
    \right)
    \cos\left(2\pi\sigma(x+r_+t)\right) \notag\\
    &\quad
    -\frac{2}{\sigma}
    \partial_x\left(
    P_{\leq\tau}(a_{(+)})(t,x)
    P_{\leq\tau}(\phi_{(r)})(x)
    \right)
    \partial_x\cos\left(2\pi\sigma(x+r_+t)\right) \notag\\
    &\quad
    +\frac{2}{\sigma}
    \partial_t\left(P_{\leq\tau}(a_{(+)})\right)(t,x)
    P_{\leq\tau}(\phi_{(r)})(x)
    \partial_t\cos\left(2\pi\sigma(x+r_+t)\right) \notag\\
    &\quad
    +\frac{1}{\sigma}
    P_{\leq\tau}(a_{(+)})(t,x)
    P_{\leq\tau}(\phi_{(r)})(x)
    (\partial_{tt}-\partial_{xx})
    \cos\left(2\pi\sigma(x+r_+t)\right) \notag\\
    &\quad
    +\frac{1}{\sigma}
    \partial_{tt}\left(P_{\leq\tau}(a_{(-)})\right)(t,x)
    P_{\leq\tau}(\phi_{(r)})(x)
    \cos\left(2\pi\sigma(x+r_-t)\right) \notag\\
    &\quad
    -\frac{1}{\sigma}
    \partial_{xx}\left(
    P_{\leq\tau}(a_{(-)})(t,x)
    P_{\leq\tau}(\phi_{(r)})(x)
    \right)
    \cos\left(2\pi\sigma(x+r_-t)\right) \notag\\
    &\quad
    -\frac{2}{\sigma}
    \partial_x\left(
    P_{\leq\tau}(a_{(-)})(t,x)
    P_{\leq\tau}(\phi_{(r)})(x)
    \right)
    \partial_x\cos\left(2\pi\sigma(x+r_-t)\right) \notag\\
    &\quad
    +\frac{2}{\sigma}
    \partial_t\left(P_{\leq\tau}(a_{(-)})\right)(t,x)
    P_{\leq\tau}(\phi_{(r)})(x)
    \partial_t\cos\left(2\pi\sigma(x+r_-t)\right) \notag\\
    &\quad
    +\frac{1}{\sigma}
    P_{\leq\tau}(a_{(-)})(t,x)
    P_{\leq\tau}(\phi_{(r)})(x)
    (\partial_{tt}-\partial_{xx})
    \cos\left(2\pi\sigma(x+r_-t)\right) \notag\\
    &:=
    L_{+,1}+L_{+,2}+L_{+,3}+L_{+,4}+L_{+,5}
    +L_{-,1}+L_{-,2}+L_{-,3}+L_{-,4}+L_{-,5}.
    \label{eq:wave-operator-decomposition}
\end{align}
Moreover,
we can use the identities 
\begin{align}
    \partial_x\cos\left(2\pi\sigma(x+r_\pm t)\right)
    &=
    -2\pi\sigma
    \sin\left(2\pi\sigma(x+r_\pm t)\right),\\
    \partial_t\cos\left(2\pi\sigma(x+r_\pm t)\right)
    &=
    -2\pi\sigma r_\pm
    \sin\left(2\pi\sigma(x+r_\pm t)\right),\\
    (\partial_{tt}-\partial_{xx})
    \cos\left(2\pi\sigma(x+r_\pm t)\right)
    &=
    4\pi^2\sigma^2(1-r_\pm^2)
    \cos\left(2\pi\sigma(x+r_\pm t)\right).
\end{align}
so we can get 
\begin{align}
    L_{+,3}
    &=
    4\pi
    \partial_x\left(
    P_{\leq\tau}(a_{(+)})(t,x)
    P_{\leq\tau}(\phi_{(r)})(x)
    \right)
    \sin\left(2\pi\sigma(x+r_+t)\right),\\
    L_{+,4}
    &=
    -4\pi r_+
    \partial_t\left(P_{\leq\tau}(a_{(+)})\right)(t,x)
    P_{\leq\tau}(\phi_{(r)})(x)
    \sin\left(2\pi\sigma(x+r_+t)\right),\\
    L_{+,5}
    &=
    4\pi^2\sigma(1-r_+^2)
    P_{\leq\tau}(a_{(+)})(t,x)
    P_{\leq\tau}(\phi_{(r)})(x)
    \cos\left(2\pi\sigma(x+r_+t)\right).
\end{align}
Similarly, we have
\begin{align}
    L_{-,3}
    &=
    4\pi
    \partial_x\left(
    P_{\leq\tau}(a_{(-)})(t,x)
    P_{\leq\tau}(\phi_{(r)})(x)
    \right)
    \sin\left(2\pi\sigma(x+r_-t)\right),\\
    L_{-,4}
    &=
    -4\pi r_-
    \partial_t\left(P_{\leq\tau}(a_{(-)})\right)(t,x)
    P_{\leq\tau}(\phi_{(r)})(x)
    \sin\left(2\pi\sigma(x+r_-t)\right),\\
    L_{-,5}
    &=
    4\pi^2\sigma(1-r_-^2)
    P_{\leq\tau}(a_{(-)})(t,x)
    P_{\leq\tau}(\phi_{(r)})(x)
    \cos\left(2\pi\sigma(x+r_-t)\right).
\end{align}
For $L_{+,1}$,
using the embedding $L_x^1\hookrightarrow H_x^{-s}$ for $s>\frac12$, 
we obtain
\begin{align}
    \|L_{+,1}\|_{C_tH_x^{-s}}
    &\lesssim
    \|L_{+,1}\|_{C_tL_x^1} \notag\\
    &\lesssim
    \sigma^{-1}
    \left\|
    \partial_{tt}P_{\leq\tau}\left(a_{(+)}\right)
    \right\|_{C_tL_x^\infty}
    \left\|
    P_{\leq\tau}\left(\phi_{(r)}\right)
    \right\|_{L_x^1} \notag\\
    &\lesssim
    \sigma^{-1}
    \|a_{(+)}\|_{C_{t,x}^2}
    \|\phi_{(r)}\|_{L_x^1} \notag\\
    &\lesssim
    \sigma^{-1}\ell_q^{-3}r^{\frac12}.
    \label{eq:L-plus-1-estimate}
\end{align}
For $L_{+,2}$,
we similarly use the embedding $L_x^1\hookrightarrow H_x^{-s}$ and the product rule to obtain
\begin{align}
    \|L_{+,2}\|_{C_tH_x^{-s}}
    &\lesssim
    \|L_{+,2}\|_{C_tL_x^1} \notag\\
    &\lesssim
    \sigma^{-1}
    \left\|
    \partial_{xx}\left(
    P_{\leq\tau}\left(a_{(+)}\right)
    P_{\leq\tau}\left(\phi_{(r)}\right)
    \right)
    \right\|_{C_tL_x^1} \notag\\
    &\lesssim
    \sigma^{-1}
    \left\|
    \partial_{xx}P_{\leq\tau}\left(a_{(+)}\right)
    \right\|_{C_tL_x^\infty}
    \left\|
    P_{\leq\tau}\left(\phi_{(r)}\right)
    \right\|_{L_x^1} \notag\\
    &\quad
    +
    2\sigma^{-1}
    \left\|
    \partial_xP_{\leq\tau}\left(a_{(+)}\right)
    \right\|_{C_tL_x^\infty}
    \left\|
    \partial_xP_{\leq\tau}\left(\phi_{(r)}\right)
    \right\|_{L_x^1} \notag\\
    &\quad
    +
    \sigma^{-1}
    \left\|
    P_{\leq\tau}\left(a_{(+)}\right)
    \right\|_{C_tL_x^\infty}
    \left\|
    \partial_{xx}P_{\leq\tau}\left(\phi_{(r)}\right)
    \right\|_{L_x^1} \notag\\
    &\lesssim
    \sigma^{-1}
    \left(
    \ell_q^{-3}r^{\frac12}
    +
    \ell_q^{-2}\lambda r^{\frac12}
    +
    \ell_q^{-1}\lambda^2r^{\frac12}
    \right) \notag\\
    &\lesssim
   \ell_q^{-1}\sigma^{-1}\lambda^2r^{\frac12}.
    \label{eq:L-plus-2-estimate}
\end{align}
For $L_{+,3}$,
since the spatial frequency of the product with
$\sin\left(2\pi\sigma(x+r_+t)\right)$
is comparable to $\sigma$, 
we obtain
\begin{align}
    \|L_{+,3}\|_{C_tH_x^{-s}}
    &\lesssim
    \sigma^{-s}
    \left\|
    \partial_x\left(
    P_{\leq\tau}\left(a_{(+)}\right)
    P_{\leq\tau}\left(\phi_{(r)}\right)
    \right)
    \right\|_{C_tL_x^2} \notag\\
    &\lesssim
    \sigma^{-s}
    \left\|
    \partial_xP_{\leq\tau}\left(a_{(+)}\right)
    \right\|_{C_tL_x^\infty}
    \left\|
    P_{\leq\tau}\left(\phi_{(r)}\right)
    \right\|_{L_x^2} \notag\\
    &\quad
    +
    \sigma^{-s}
    \left\|
    P_{\leq\tau}\left(a_{(+)}\right)
    \right\|_{C_tL_x^\infty}
    \left\|
    \partial_xP_{\leq\tau}\left(\phi_{(r)}\right)
    \right\|_{L_x^2} \notag\\
    &\lesssim
    \sigma^{-s}
    \left(
    \ell_q^{-2}
    +
    \ell_q^{-1}\lambda
    \right) \notag\\
    &\lesssim
    \sigma^{-s}\ell_q^{-1}\lambda.
    \label{eq:L-plus-3-estimate}
\end{align}
For $L_{+,4}$, similarly 
we have
\begin{align}
    \|L_{+,4}\|_{C_tH_x^{-s}}
    &\lesssim
    \sigma^{-s}
    \left\|
    \partial_tP_{\leq\tau}\left(a_{(+)}\right)
    P_{\leq\tau}\left(\phi_{(r)}\right)
    \right\|_{C_tL_x^2} \notag\\
    &\lesssim
    \sigma^{-s}
    \left\|
    \partial_tP_{\leq\tau}\left(a_{(+)}\right)
    \right\|_{C_tL_x^\infty}
    \left\|
    P_{\leq\tau}\left(\phi_{(r)}\right)
    \right\|_{L_x^2} \notag\\
    &\lesssim
    \sigma^{-s}
    \|a_{(+)}\|_{C_{t,x}^1}
    \|\phi_{(r)}\|_{L_x^2} \notag\\
    &\lesssim
    \sigma^{-s}\ell_q^{-2}.
    \label{eq:L-plus-4-estimate}
\end{align}
For $L_{+,5}$, we can get 
\begin{align}
    \|L_{+,5}\|_{C_tH_x^{-s}}
    &\lesssim
    \sigma^{1-s}
    \left\|
    P_{\leq\tau}\left(a_{(+)}\right)
    \right\|_{C_tL_x^\infty}
    \left\|
    P_{\leq\tau}\left(\phi_{(r)}\right)
    \right\|_{L_x^2} \notag\\
    &\lesssim
    \sigma^{1-s}
    \|a_{(+)}\|_{C_{t,x}}
    \|\phi_{(r)}\|_{L_x^2} \notag\\
    &\lesssim
    \sigma^{1-s}\ell_q^{-1}.
    \label{eq:L-plus-5-estimate}
\end{align}
where the estimates of $ L_{-,1}+L_{-,2}+L_{-,3}+L_{-,4}+L_{-,5}$ can follow that.
It remains to estimate the interaction terms. We denote 
\begin{align}
    L_2
    &:=
    2\Theta_1\,\partial_tu_q\,\partial_tw_{q+1},\\
    L_3
    &:=
    2\Theta_2\,\partial_xu_q\,\partial_xw_{q+1}.
\end{align}
By the embedding theorem, 
we obtain
\begin{align}
    \| L_2 \|_{C_tL_x^1}
    &\lesssim
    \|\partial_tu_q\|_{C_tL_x^\infty}
    \|\partial_tw_{q+1}\|_{C_tL_x^1}\notag\\
    &\lesssim
    \ell_q^{-2}r^{\frac12}.
    \label{eq:linear-interaction-time}
\end{align}
Similarly, we have
\begin{align}
    \|L_3\|_{C_tL_x^{1}}
     &\lesssim
    \ell_q^{-2}r^{\frac1p-\frac12}.
    \label{eq:linear-interaction-space}
\end{align}
Combining \eqref{eq:L-plus-1-estimate}--\eqref{eq:L-plus-5-estimate} and \eqref{eq:linear-interaction-time}--\eqref{eq:linear-interaction-space}, we finally obtain
\begin{align}
    \|E_{q+1}^{\mathrm{lin}}\|_{C_tH_x^{-s}}
    &\lesssim
    \sigma^{1-s}\ell_q^{-1}.
    \label{eq:linear-error-final}
\end{align}

Finally, for the corrector error, by H\"older's inequality and Lemma \ref{lem:estimates-wq1}, we have
\begin{align}
    &\|E_{q+1}^{\mathrm{cor}}\|_{C_t\Hs}\lesssim \|E_{q+1}^{\mathrm{cor}}\|_{C_tL_x^1} \lesssim
    \left\|
    \partial_t w_{q+1}^{(p)}
    \partial_t w_{q+1}^{(c)}
    \right\|_{C_tL_x^1}
    +
    \left\|
    \left(\partial_t w_{q+1}^{(c)}\right)^2
    \right\|_{C_tL_x^1} \notag\\
    &\quad
    +
    \left\|
    \partial_x w_{q+1}^{(p)}
    \cdot
    \partial_x w_{q+1}^{(c)}
    \right\|_{C_tL_x^1}
    +
    \left\|
    \left(\partial_x w_{q+1}^{(c)}\right)^2
    \right\|_{C_tL_x^1} \notag\\
    &\lesssim
    \|\partial_t w_{q+1}^{(p)}\|_{C_tL_x^2}
    \|\partial_t w_{q+1}^{(c)}\|_{C_tL_x^2}
    +
    \|\partial_t w_{q+1}^{(c)}\|_{C_tL_x^2}^2 \notag\\
    &\quad
    +
    \|\partial_x w_{q+1}^{(p)}\|_{C_tL_x^2}
    \|\partial_x w_{q+1}^{(c)}\|_{C_tL_x^2}
    +
    \|\partial_x w_{q+1}^{(c)}\|_{C_tL_x^2}^2 \notag\\
    &\lesssim
    \ell_q^{-1}
    \left(\tau^{-1}\lambda\ell_q^{-2}\right)
    +
    \left(\tau^{-1}\lambda\ell_q^{-2}\right)^2 \notag\\
    &\lesssim
    \tau^{-1}\lambda\ell_q^{-3}
    +
    \tau^{-2}\lambda^2\ell_q^{-4}
    \lesssim  \tau^{-1}\lambda\ell_q^{-3}.
    \label{eq:corrector-error-estimate}
\end{align}
This completes the proof. 
\end{proof}
\subsection{Proof of main iteration Lemma \ref{lem:iteration}}
   Since 
$
    u_{q+1}=u_q+w_{q+1},
$
by \eqref{eq:uq-CN} and \eqref{eq:w-CN-new}, we have
\begin{align}
    \|u_{q+1}\|_{C^N_{t,x}}
    &\leq
    \|u_q\|_{C^N_{t,x}}
    +
    \|w_{q+1}\|_{C^N_{t,x}}
    \notag\\
    &\lesssim
    \lambda_q^{6N+6}
    +
     \ell_q^{-1}\sigma^{N-1}r^{-\frac12}
    \notag\\
    &\lesssim
    \lambda_{q+1}^{6N+6},
\end{align}
which proves \eqref{eq1:uq-CN}.

\eqref{eq1:inductive-defect} follows directly from Lemma \ref{lemma:hs}.

For \eqref{eq1:inductive-defect-CN}, 
we have
\begin{align}
    \square u_q
    =
    \Theta_1(\partial_tu_q)^2
    +
    \Theta_2(\partial_x u_q)^2
    +
    E_q,
\end{align}
where $\square=\partial_{tt}-\partial_{xx}$.
Since $u_{q+1}=u_q+w_{q+1}$,
we obtain
\begin{align}
    E_{q+1}
    &=
    \square u_{q+1}
    -
    \Theta_1(\partial_tu_{q+1})^2
    -
    \Theta_2(\partial_x u_{q+1})^2 \notag\\
    &=
    E_q
    +
    \square w_{q+1}
    -
    2\Theta_1\,\partial_tu_q\partial_tw_{q+1}
    -
    2\Theta_2\,\partial_x u_q\cdot\partial_x w_{q+1} \notag\\
    &\quad
    -
    \Theta_1(\partial_tw_{q+1})^2
    -
    \Theta_2(\partial_x w_{q+1})^2.
    \label{eq:Eq1-CN}
\end{align}
By using \eqref{eq:inductive-defect-CN} and \eqref{eq:w-CN-new},  
we have
\begin{align}
    \|E_{q+1}\|_{C_{t,x}^N}
    &\lesssim
    \|E_q\|_{C_{t,x}^N}
    +
    \|\square w_{q+1}\|_{C_{t,x}^N} \notag\\
    &\quad
    +
    \|\partial_tu_q\partial_tw_{q+1}\|_{C_{t,x}^N}
    +
    \|\partial_x u_q\cdot\partial_x w_{q+1}\|_{C_{t,x}^N} \notag\\
    &\quad
    +
    \|(\partial_tw_{q+1})^2\|_{C_{t,x}^N}
    +
    \|(\partial_x w_{q+1})^2\|_{C_{t,x}^N} \notag\\
    &\lesssim
    \|E_q\|_{C_{t,x}^N}
    +
    \|w_{q+1}\|_{C_{t,x}^{N+2}} \notag\\
    &\quad
    +
    \|u_q\|_{C_{t,x}^{N+1}}
    \|w_{q+1}\|_{C_{t,x}^{N+1}}
    +
    \|w_{q+1}\|_{C_{t,x}^{N+1}}^2 \notag \\
    &\lesssim
    \lambda_q^{12N+12}
    +
    \ell_q^{-1}\sigma^{N+1}r^{-\frac12} \notag\\
    &\quad
    +
    \lambda_q^{6N+6}
    \ell_q^{-1}\sigma^Nr^{-\frac12}
    +
    \ell_q^{-2}\sigma^{2N}r^{-1}\notag\\
    &\lesssim
    \lambda_{q+1}^{12N+12}.
    \label{eq:Eq1-CN-final}
\end{align}

For \eqref{eq1:time-support-q}, by the inductive assumption \eqref{eq:time-support-q}, we have 
\[
\operatorname{supp}_tE_q\subset[t_q,T].
\]
And we can use the definition of $\chi_q(t)$ to get 
\[
\operatorname{supp}_t w_{q+1}
\subset[t_q-\ell_q/2,T].
\]
Since \(t_{q+1}=t_q-\ell_q<t_q-\ell_q/2\), it follows that
\[
w_{q+1}=0\qquad\text{on }[0,t_{q+1}]\times\mathbb T.
\]
Since \(u_{q+1}=u_q+w_{q+1}\), we conclude that
\[
u_{q+1}=0,\quad E_{q+1}=0\qquad\text{on }[0,t_{q+1}]\times\mathbb T,
\]
which proves \eqref{eq1:time-support-q}.

For \eqref{eq1:space-support-q}--\eqref{item:prop-hatw-supp}, we can use Lemma \ref{lemma:supp}.

For \eqref{item:prop-w-Lp}, we can get it by using \eqref{eq:w-Cb-new}.

For \eqref{item:prop-w-paraproduct}, we can calculate as 
\begin{align}
    \|Q(w_{q+1},w_{q+1})\|_{C_t\Hs} &\lesssim \|Q(w_{q+1},w_{q+1}) - E_q\|_{C_t^0\Hs} + \|E_q\|_{C_t^0\Hs}\notag\\
    &\leq \|E_{\mathrm{osc}} + E_{\mathrm{cor}}\|_{C_t\Hs} + \|E_q\|_{C_t\Hs} \lesssim \delta_{q+1}.
    \label{es:product1}
\end{align}
By the Sobolev embedding, we can get by using \eqref{eq:uq-CN}, \eqref{eq:wp-derivative-new} and \eqref{eq:wc-derivative-new} that 
\begin{align}
    &\|Q(P_Mu_q,w_{q+1})\|_{C_t\Hs}+\|Q(w_{q+1},P_Mu_q)\|_{C_t\Hs}\notag\\
    &\lesssim \|P_M\partial_t u_q\cdot\partial_t w_{q+1}\|_{C_tL_x^1} +\|P_M\partial_x u_q\cdot\partial_x w_{q+1}\|_{C_tL_x^1}\notag\\
    &\lesssim \|P_M\partial_t u_q\|_{C_{t,x}}\|\partial_t w_{q+1}\|_{C_tL_x^1} + \|P_M\partial_x u_q\|_{C_{t,x}}\|\partial_x w_{q+1}\|_{C_tL_x^1}\notag\\
    &\lesssim \delta_{q+2}.
\label{es:product2}
\end{align}
Combining \eqref{es:product1} and \eqref{es:product2}, we can obtain \eqref{item:prop-w-paraproduct}. 
This completes the proof.

\section{Proof of the main theorem}\label{sec5}
Fix $0<\beta<1$ and $s>1$ as in Lemma \ref{lem:iteration}.
We divide the proof into three steps.

\medskip
\noindent
\textbf{Step 1: Construction of the approximate solution sequences.}

We check the iteration step for $q=0$. Given $t_0=T/4$ and choose a function
$\eta\in C_c^\infty((t_0,T))$ such that
\[
    \eta(t)=0
    \qquad\text{for }0\leq t\leq t_0.
\]
Let
\[
    \psi(x):=\cos(2\pi x),
\]
and choose infinitely many distinct constants
\[
    \kappa_i\in(0,\kappa_*),
    \qquad i\in\mathbb N,
\]
where $\kappa_*>0$ will be chosen sufficiently small.

For each $i\in\mathbb N$, define
\begin{equation}\label{eq:initial-subsolution-i}
    u_0^{(i)}(t,x)
    :=
    \kappa_i\eta(t)\psi(x).
\end{equation}
Then
\begin{equation}\label{eq:common-zero-data}
    u_0^{(i)}(0,\cdot)=0,
    \qquad
    \partial_tu_0^{(i)}(0,\cdot)=0
\end{equation}
for every $i$.

Define the corresponding initial error by
\begin{equation}\label{eq:E0-i-definition}
    E_0^{(i)}
    :=
    \partial_{tt}u_0^{(i)}
    -
    \partial_{xx} u_0^{(i)}
    -
    \Theta_1(\partial_t u_0^{(i)})^2
    -
    \Theta_2(\partial_x u_0^{(i)})^2.
\end{equation}
By a direct calculation, we can get 
\begin{align}
    E_0^{(i)}
    ={}&\kappa_i \left(\eta''(t)\psi-\eta(t)\psi''\right)
    -\kappa_i^2\eta(t)'^2\psi^2
    -\kappa_i^2\eta(t)^2(\partial_x\psi)^2.
    \label{eq:E0-explicit}
\end{align}
Therefore, by choosing $\kappa_*>0$ sufficiently small, we may ensure
that
\begin{equation}\label{eq:E0-small}
    \|E_0^{(i)}\|_{C_t^0\Hs}
    \lesssim\delta_1
\end{equation}
uniformly in $i$.

Since $\psi$ has finite Fourier support, for large dyadic number $a$ with 
$\lambda_0^{(i)} = a ^{b^0} =a$, we always have 
\[
    \operatorname{supp}_k\widehat{u_0^{(i)}}
    \subset B(0,\lambda_0^5)
\]
for every $i$. 

Let 
\[
    w_0 := u_0^{(i)},
\]
we can know that $(u_0^{(i)},E_0^{(i)})$ is a smooth solution to \eqref{eq:approx-system}.
For every $i\in\mathbb N$, we obtain a
sequence of smooth approximate solutions
\[
    \{(u_q^{(i)},E_q^{(i)})\}_{q\geq0}
\]
satisfying \eqref{eq:approx-system}. 
By Lemma~\ref{lem:iteration},
\begin{equation}\label{eq:w-beta-summability}
    \|w_{q+1}^{(i)}\|_{C_{t,x}^\beta}
    +
    \|\partial_tw_{q+1}^{(i)}\|_
        {C_{t}B_{\infty,\infty}^{\beta-1}}
    +
     \|w_{q+1}^{(i)}\|_{C_{t}W_x^{1,p_*}}
    \lesssim
    \delta_{q+2},
\end{equation}
we can get that 
the sequence $\{u_q^{(i)}\}_{q\geq0}$ is a Cauchy sequence in
$C_{t,x}^\beta$ and ${C_{t}W_x^{1,p_*}}$, also 
$\{\partial_tu_q^{(i)}\}_{q\geq0}$ is a Cauchy sequence in
$C_{t}B_{\infty,\infty}^{\beta-1}$. Hence there exists a limit $u^{(i)}$ such that
\begin{align}
    u_q^{(i)}
    &\longrightarrow u^{(i)}
    &&\text{in }C_{t,x}^{\beta},
    \label{eq:uqi-limit}\\
    \partial_tu_q^{(i)}
    &\longrightarrow v^{(i)}
    &&\text{in }C_{t}B_{\infty,\infty}^{\beta-1},
    \label{eq:dt-uqi-limit}\\
    u_q^{(i)}
    &\longrightarrow u^{(i)}
    &&\text{in }C_{t}W_x^{1,p_*}.
    \label{eq:w1p-uqi-limit}
\end{align}
For $\varphi\in C_c^{\infty}([0,T]\times\mathbb{T})$, we have  
\begin{align}
\langle v^{(i)},\varphi\rangle
&=
\lim_{q\to\infty}\langle \partial_tu_q^{(i)},\varphi\rangle
=
-\lim_{q\to\infty}\langle u_q^{(i)},\partial_t\varphi\rangle
=
-\langle u^{(i)},\partial_t\varphi\rangle,
\end{align}
hence,
$
v^{(i)}=\partial_tu^{(i)}
$
in the sense of distributions.
Moreover, we can get 
\begin{equation}\label{eq:limit-regularity-i}
   u^{(i)}\in
	C_{t,x}^{\beta}([0,T]\times\mathbb T) \cap C_t^1 ([0,T], B_{\infty,\infty}^{\beta-1}(\mathbb T))\cap C_{t}([0,T],  W^{1,p_*}(\mathbb T)) .
\end{equation}
Moreover, by the iteration Lemma \ref{lem:iteration}, we can get 
\begin{equation}\label{eq:limit-zero-data}
    u^{(i)}(0,\cdot)=0,
    \qquad
    \partial_tu^{(i)}(0,\cdot)=0.
\end{equation}

\medskip

\noindent
\textbf{Step 2: Verification of the products in $H^{-s}$.}

Now we show that $Q(u,u)$ is well-defined in
$C_t \Hs$.
For simplicity, we omit the superscript $(i)$ in this step.
Since
\[
    \operatorname{supp}_k\widehat u_q
    \subset B(0,\lambda_q^5),
\]
we can get 
\begin{align}
    P_{M_j}\partial_t w_l
    &=
    \begin{cases}
        0, & l\neq j,\\
        \partial_t w_j, & l=j,
    \end{cases}
    \qquad
    P_{M_j}\partial_x w_l
    =
    \begin{cases}
        0, & l\neq j,\\
        \partial_x w_j, & l=j,
    \end{cases}
    \label{eq:frequency-separation-new}
\end{align}
where $M_j$ be a dyadic frequency corresponding to the frequency shell
of $w_j$.
By \eqref{item:prop-w-paraproduct} in Lemma \ref{lem:iteration} and the frequency localization
property of $w_{q+1}$, we check as
\begin{align}
    &\sum_{1\leq M,M'\leq M_{\max}}
    \left\| Q(P_M u,P_{M'}u)
    \right\|_{C_t^0H_x^{-s}}\notag\\
    &\lesssim
    \sum_{2M_{\max}\geq q,q'\geq0}
    \left\|
    Q(w_{q},w_{q'})
    \right\|_{C_t^0H_x^{-s}} \notag\\
    &=
    \sum_{2Q_{\max}\geq q\geq0}
    \left\|
    Q(w_{q},w_{q})
    \right\|_{C_t^0H_x^{-s}} 
    +
    2\sum_{0\leq q'<q\leq2Q_{\max}}
    \left\|
    Q(w_{q},w_{q'})
    \right\|_{C_t^0H_x^{-s}} \notag\\
    &\lesssim
    \sum_{q\geq0}\delta_{q}
    <\infty.
\end{align}
So $Q(u,u)$ is well-defined as 
\begin{equation}
    Q(u,u)
    :=
    \lim_{q\to\infty}Q(u_q,u_q)
    \qquad\text{in }C_t\Hs.
    \label{eq:Q-limit-equals-LP-product}
\end{equation}

\medskip

\noindent
\textbf{Step 3: Get infinitely many singular weak solutions with zero initial data.}

First, we check that $u^{(i)}$ is a singular weak solution to \eqref{eq:NLW}. 
By using \eqref{eq:inductive-defect}, we have 
\begin{equation}\label{eq:error-goes-zero}
    E_q^{(i)}
    \longrightarrow0
    \qquad
    \text{in }C_t\Hs.
\end{equation}
For every
$
    \varphi\in
    C_c^\infty([0,T)\times\mathbb T),
$
by \eqref{eq:approx-system}, 
the smooth approximate solution $u_q^{(i)}$ satisfies
\begin{align}
    &\int_0^T
    \left\langle u_q^{(i)},\partial_{tt}\varphi-\partial_{xx}\varphi\right\rangle\,dt -
    \int_0^T\left\langle  Q(u_q^{(i)},u_q^{(i)}), \varphi\right\rangle\,dt
    =\int_0^T\left\langle E_q^{(i)},\varphi \right\rangle\,dt.
    \label{eq:approx-weak-i}
\end{align}
Taking the limit $q\to\infty$ in \eqref{eq:approx-weak-i}, 
we obtain that  $u^{(i)}$ satisfies \eqref{eq:NLW} in $\mathscr D'((0,T)\times\mathbb T)$ 
combining with \eqref{eq:limit-zero-data}.

Let $\kappa_i=2^{-i}\kappa_*$, for every $i<j$, we have
\begin{align}
    \left\|
    u_0^{(i)}(t_*,\cdot)
    -
    u_0^{(j)}(t_*,\cdot)
    \right\|_{C_x^\beta}
    =
    |\kappa_i-\kappa_j|
    \|\psi\|_{C_x^\beta}
    \geq
    \frac12\kappa_i
    \|\psi\|_{C_x^\beta},
    \label{eq:initial-subsolution-separation}
\end{align}
where $t_* \in (t_0,T)$ satisfies $\eta(t_*) \neq 0$.
For each $i\in\mathbb N$, we can choose the parameter $a$ sufficiently large so that
\begin{equation}
    \sum_{q=0}^{\infty}
    \delta_{q+2}^{(i)}
    <
    \frac{1}{16}
    \kappa_i
    \|\psi\|_{C_x^\beta}.
    \label{eq:iteration-tail-relative-small}
\end{equation}
By using Lemma~\ref{lem:iteration}, we obtain 
\begin{align}
    \left\|
    u^{(i)}-u_0^{(i)}
    \right\|_{C_{t,x}^\beta}
    &\leq
    \sum_{q=0}^{\infty}
    \left\|
    w_{q+1}^{(i)}
    \right\|_{C_{t,x}^\beta}\notag\\
    &\lesssim
    \sum_{q=0}^{\infty}
    \delta_{q+2}^{(i)}
    <
    \frac{1}{16}
    \kappa_i
    \|\psi\|_{C_x^\beta}.
    \label{eq:limit-close-initial-subsolution}
\end{align}
Hence, we have
\begin{align}
    \left\| u^{(i)} - u^{(j)}
    \right\|_{C_{t,x}^\beta}
    &\geq
    \left\| u_0^{(i)}-u_0^{(j)}
    \right\|_{C_{x}^\beta}
    -
    \left\|u^{(i)} - u_0^{(i)}
    \right\|_{C_{t,x}^\beta}
    -
    \left\|
    u^{(j)}-u_0^{(j)}
    \right\|_{C_{t,x}^\beta}\notag\\
    &\geq
    \left(
        \frac12-\frac1{16}-\frac1{16}
    \right)
    \kappa_i
    \|\psi\|_{C_{x}^\beta}\notag\\
    &=
    \frac{3}{8}
    \kappa_i
    \|\psi\|_{C_x^\beta}
    >0.
    \label{eq:distinct-limit-solutions}
\end{align}
Therefore, $u^{(i)}\neq u^{(j)}$ whenever $i\neq j$.

This completes the proof of Theorem~\ref{thm1}.
\appendix\label{appendix1}

\section{Proof of the uniqueness of $C^1$ weak solutions}\label{app:auxiliary-lemma}
\begin{proposition}[Uniqueness of $C^1$ weak solutions]
Let $\Theta_1,\Theta_2\in\mathbb R$,
and let $u,v\in C^1([0,T]\times\mathbb T)$ be two weak solutions of
\begin{align}
    \partial_{tt}u-\partial_{xx}u
    =
     \Theta_1(\partial_tu)^2+ \Theta_2(\partial_xu)^2,\notag
\end{align}
with the same initial data
\begin{align}
    u(0,\cdot)=v(0,\cdot),
    \qquad
    \partial_tu(0,\cdot)=\partial_tv(0,\cdot).\notag
\end{align}
Then
\begin{align}
    u=v
    \qquad
    \text{on }[0,T]\times\mathbb T.\notag
\end{align}
In particular,
the weak solution is unique in
$C^1([0,T]\times\mathbb T)$.
\end{proposition}
\begin{proof}
Following the argument of Sideris \cite{Sideris1981Thesis}, we set 
\begin{align}
    w:=u-v.
\end{align}
Since $u$ and $v$ satisfy the same equation, we obtain in the sense of distributions
\begin{align}
    \partial_{tt}w-\partial_{xx}w
    &=
    \Theta_1
    \left(
        (\partial_tu)^2-(\partial_tv)^2
    \right)
    +
    \Theta_2
    \left(
        (\partial_xu)^2-(\partial_xv)^2
    \right)\notag\\
    &=
    \Theta_1
    (\partial_tu+\partial_tv)\partial_tw
    +
    \Theta_2
    (\partial_xu+\partial_xv)\partial_xw.
    \label{eq:difference-wave}
\end{align}
For simplicity, we denote the right-hand side by
\begin{align}
    F(t,x)
    &:=
    \Theta_1
    (\partial_tu+\partial_tv)\partial_tw
    +
    \Theta_2
    (\partial_xu+\partial_xv)\partial_xw.
    \label{eq:difference-nonlinearity}
\end{align}
Since $u,v\in C^1([0,T]\times\mathbb T)$, we know that $F\in C([0,T]\times\mathbb T)$. Moreover, the two solutions have the same initial data, so that
\begin{align}
    w(0,\cdot)=0,
    \qquad
    \partial_tw(0,\cdot)=0.
    \label{eq:difference-zero-data}
\end{align}
By the Duhamel formula for the one-dimensional wave equation, after periodically extending the functions from $\mathbb T$ to $\mathbb R$, we have
\begin{align}
    w(t,x)
    =
    \frac12
    \int_0^t
    \int_{x-(t-s)}^{x+(t-s)}
    F(s,y)\,dy\,ds.
    \label{eq:difference-duhamel}
\end{align}
Differentiating \eqref{eq:difference-duhamel} with respect to $t$ and $x$, we obtain
\begin{align}
    \partial_tw(t,x)
    &=
    \frac12
    \int_0^t
    \left[
        F(s,x+t-s)
        +
        F(s,x-t+s)
    \right]\,ds,
    \label{eq:duhamel-time-derivative}\\
    \partial_xw(t,x)
    &=
    \frac12
    \int_0^t
    \left[
        F(s,x+t-s)
        -
        F(s,x-t+s)
    \right]\,ds.
    \label{eq:duhamel-space-derivative}
\end{align}
Therefore,
\begin{align}
    \|\partial_tw(t,\cdot)\|_{L_x^\infty}
    +
    \|\partial_xw(t,\cdot)\|_{L_x^\infty}
    \lesssim
    \int_0^t
    \|F(s,\cdot)\|_{L_x^\infty}\,ds.
    \label{eq:difference-derivative-estimate}
\end{align}
On the other hand, by the definition of $F$, we have
\begin{align}
    \|F(t,\cdot)\|_{L_x^\infty}
    &\leq
    |\Theta_1|
    \|\partial_tu+\partial_tv\|_{L_x^\infty}
    \|\partial_tw\|_{L_x^\infty}\notag\\
    &\quad+
    |\Theta_2|
    \|\partial_xu+\partial_xv\|_{L_x^\infty}
    \|\partial_xw\|_{L_x^\infty}\notag\\
    &\lesssim
    C_{u,v}
    \left(
        \|\partial_tw\|_{L_x^\infty}
        +
        \|\partial_xw\|_{L_x^\infty}
    \right),
    \label{eq:difference-nonlinearity-estimate}
\end{align}
where
\begin{align}
    C_{u,v}
    &:=
    |\Theta_1|
    \left(
        \|\partial_tu\|_{C_{t,x}}
        +
        \|\partial_tv\|_{C_{t,x}}
    \right)
    +
    |\Theta_2|
    \left(
        \|\partial_xu\|_{C_{t,x}}
        +
        \|\partial_xv\|_{C_{t,x}}
    \right)
    <\infty.
\end{align}
Now, we define
\begin{align}
    M(t)
    &:=
    \|\partial_tw(t,\cdot)\|_{L_x^\infty}
    +
    \|\partial_xw(t,\cdot)\|_{L_x^\infty}.
\end{align}
Combining \eqref{eq:difference-derivative-estimate} and
\eqref{eq:difference-nonlinearity-estimate}, we can get
\begin{align}
    M(t)
    \lesssim
    C_{u,v}
    \int_0^t
    M(s)\,ds.
\end{align}
By Gronwall's inequality, we conclude that
\begin{align}
    M(t)=0
    \qquad
    \text{for every }t\in[0,T].
\end{align}
Hence
\begin{align}
    \partial_tw=\partial_xw=0
    \qquad
    \text{on }[0,T]\times\mathbb T.
\end{align}
Since $w(0,\cdot)=0$, we finally obtain
\begin{align}
    w\equiv0
    \qquad
    \text{on }[0,T]\times\mathbb T.
\end{align}
Therefore, $u=v$, which completes the proof.
\end{proof}

\section*{Acknowledgments}
This research is supported in part by NSFC Grants No.12431007 and 62588101. 
The authors would like to thank Professor Yi Zhou and Professor Ningan Lai for helpful discussions.

\bibliographystyle{plain}  
\bibliography{NLW_ref}  
\end{document}